\documentclass[11pt]{amsart}
\usepackage[latin9]{inputenc}
\usepackage[a4paper]{geometry}
\usepackage{color}
\usepackage{units}
\usepackage{textcomp}
\usepackage{mathrsfs}
\usepackage{amstext}
\usepackage{amsthm}
\usepackage{amssymb}
\usepackage{graphicx}
\usepackage[unicode=true,pdfusetitle,
 bookmarks=false,
 breaklinks=false,pdfborder={0 0 0},pdfborderstyle={},backref=false,colorlinks=true]
 {hyperref}
\hypersetup{
 linkcolor=blue,  citecolor=blue, urlcolor=blue}

\makeatletter

\newcommand{\noun}[1]{\textsc{#1}}
\providecommand{\tabularnewline}{\\}

\numberwithin{equation}{section}
\numberwithin{figure}{section}
\theoremstyle{plain}
\newtheorem*{thm*}{\protect\theoremname}
\theoremstyle{plain}
\newtheorem{thm}{\protect\theoremname}[section]
\theoremstyle{plain}
\newtheorem{cor}[thm]{\protect\corollaryname}
\theoremstyle{remark}
\newtheorem*{acknowledgement*}{\protect\acknowledgementname}
\theoremstyle{definition}
\newtheorem{defn}[thm]{\protect\definitionname}
\theoremstyle{plain}
\newtheorem{prop}[thm]{\protect\propositionname}
\theoremstyle{remark}
\newtheorem{claim}[thm]{\protect\claimname}

\usepackage[nobysame,abbrev,alphabetic]{amsrefs}
\usepackage[all]{xy}
\usepackage{verbatim,enumitem}
\usepackage{cancel}
\DeclareMathAlphabet{\mathbbold}{U}{bbold}{m}{n}
\usepackage{tikz}

\setlist[enumerate]{leftmargin=*,widest=0}

\DeclareMathOperator{\col}{col}

\DeclareMathOperator{\supp}{supp}

\DeclareMathOperator{\diag}{diag}
\DeclareMathOperator{\BT}{\mathcal{B}}
\DeclareMathOperator{\X}{\mathcal{X}}
\DeclareMathOperator{\dist}{\textrm{dist}}
\newcommand{\one}{\mathbbold{1}}
\newcommand{\zero}{\mathbbold{0}}
\newcommand{\Mod}[1]{\,\left(\textup{mod}\;#1\right)}

\allowdisplaybreaks[2]
\makeatother

\providecommand{\acknowledgementname}{Acknowledgement}
\providecommand{\claimname}{Claim}
\providecommand{\corollaryname}{Corollary}
\providecommand{\definitionname}{Definition}
\providecommand{\propositionname}{Proposition}
\providecommand{\theoremname}{Theorem}

\begin{document}
\global\long\def\eps{\varepsilon}%

\title[Cutoff on Ramanujan complexes and classical groups]{Cutoff on Ramanujan complexes\\
and classical groups}
\author{Michael Chapman}
\author{Ori Parzanchevski}
\begin{abstract}
The total-variation cutoff phenomenon has been conjectured to hold
for simple random walk on all transitive expanders. However, very
little is actually known regarding this conjecture, and cutoff on
sparse graphs in general. In this paper we establish total-variation
cutoff for simple random walk on Ramanujan complexes of type $\smash{\widetilde{A}_{d}}$
($d\geq1$). As a result, we obtain explicit generators for the finite
classical groups $\mathrm{PGL}_{n}(\mathbb{F}_{q})$ for which the
associated Cayley graphs exhibit total-variation cutoff.
\end{abstract}

\maketitle
\thispagestyle{empty}

\section{Introduction}

The $\varepsilon$-mixing time of a finite Markov chain is the earliest
time at which the distribution on states becomes $\varepsilon$-close
to the stationery one, regardless of the starting distribution. There
are several distance functions which one may use, and we focus on
$L^{1}$, or total-variation (see \eqref{eq:TV}). The parameter $\varepsilon$
can be thought of as a measure of ``standards'': for example, in
a professional poker tournament one expects the dealer to shuffle
the decks longer than in an amateur one. Loosely speaking, a sequence
of Markov chains is said to exhibit the \emph{cutoff phenomenon} if
it is insensitive to ones' standards. Namely, whether one seeks to
be at most $0.01$ away from the stationary distribution or at most
$0.99$ away from it, it will roughly take the same amount of time.
In other words, for a long period of time the distribution is at almost
maximal distance from stationery, and then over a short period of
time it mixes almost completely. This counter-intuitive phenomenon
was first demonstrated by Diaconis--Shahshahani and Aldous \cite{diaconis1981generating,aldous1983random},
and was subsequently shown to hold in many naturally occurring Markov
chains (see the surveys \cite{Diaconis1996cutoffphenomenonfinite,SaloffCoste2004RandomWalksFinite}).
Common to all of these examples is that the number of legal moves
grows together with the number of states.

The case of a bounded number of legal moves -- for example, simple
random walk (SRW) on a family of graphs with bounded degrees -- turned
out to be more resistant, and much less is known about it. In 2004,
Peres has conjectured that SRW on every family of transitive bounded
degree expanders exhibits the cutoff phenomenon \cite{AmirDembo2004SharpThresholdsMixing},
even though at the time \emph{no }family of bounded degree expanders
was known to do so. In \cite{Lubetzky2010Cutoffphenomenarandom},
Lubetzky and Sly used probabilistic methods to show that random regular
graphs exhibit cutoff a.a.s.; The next big breakthrough was achieved
by Lubetzky and Peres \cite{lubetzky2016cutoff}, who showed that
SRW on all Ramanujan graphs (which are optimal expanders) exhibit
cutoff. A main ingredient of \cite{lubetzky2016cutoff} is to show
first that non-backtracking random walk (NBRW) on Ramanujan graphs
exhibits cutoff at an optimal time. The last assertion was generalized
in \cite{Lubetzky2017RandomWalks} to the context of Ramanujan complexes,
which are high-dimensional analogues of Ramanujan graphs, defined
in \cite{li2004ramanujan,Lubotzky2005a}. In the paper \cite{Lubetzky2017RandomWalks},
Lubetzky, Lubotzky and the second author establish optimal-time cutoff
for a large family of \emph{asymmetric} random walks on the cells
of these complexes (in the graph case, NBRW is an asymmetric walk
on edges). However, the techniques of \cite{Lubetzky2017RandomWalks}
cannot be applied neither to any symmetric\emph{ }random walk, nor
to any walk on vertices.

The goal of the current paper is to establish cutoff for SRW on Ramanujan
complexes arising from the group $\mathrm{PGL}_{d}$ over a non-archimedean
local field. Interestingly, while SRW on vertices only ``sees''
the $1$-skeleton of the complex, our proof makes use of asymmetric
random walks on cells of \emph{all }dimensions of the complex, showing
that the high-dimensional geometry can play an important role even
when studying walks on graphs. A main motivation to study these complexes
is the study of expansion in simple groups: the finite groups $\mathrm{PGL}_{2}(\mathbb{F}_{q})$
admit a Cayley structure of a Ramanujan graph due to Lubotzky, Phillips
and Sarnak \cite{LPS88}, whereas the groups $\mathrm{PGL}_{d}(\mathbb{F}_{q})$
for general $d$ can be endowed with a Cayley structure which is the
$1$-skeleton of a Ramanujan complex. Thus, we establish here cutoff
for SRW on the groups $\mathrm{PGL}_{d}(\mathbb{F}_{q})$, with respect
to the appropriate generators. We remark that the situation in the
case $d\geq3$ is even more striking than in the graph case $(d=2)$:
by Kazhdan's property $(T)$, for $d\geq3$ every generating set of
$\mathrm{PGL}_{d}\left(\mathbb{Z}\right)$ gives rise to an expander
family of Cayley graphs of $\mathrm{PGL}_{d}(\mathbb{F}_{q})$ \cite{Margulis1973ExplicitKazhdan},
but except for the case which we handle in this paper, it is unknown
whether these families exhibit cutoff or not. \medskip{}

We now move on to rigorous terms. Let $\mathcal{D}$ be a connected
directed graph (\emph{digraph}), which we assume for simplicity to
be $k$-out and $k$-in regular. Consider random walk on $\mathcal{D}$
starting at a vertex $v_{0}$ with uniform transition probabilities,
and denote by $\mu_{\mathcal{D},v_{0}}^{t}$ its distribution at time
$t$. The $\varepsilon$-mixing time of $\mathcal{D}$ is 
\[
t_{mix}\left(\varepsilon\right)=t_{mix}\left(\varepsilon,\mathcal{D}\right)=\min\left\{ t\in\mathbb{N}\,\middle|\,\forall v_{0}\in\mathcal{D}^{0},\ \Vert\mu_{\mathcal{D},v_{0}}^{t}-\pi_{\mathcal{D}}\Vert_{TV}<\varepsilon\right\} ,
\]
where $\pi_{\mathcal{D}}$ is the uniform distribution on $\mathcal{D}^{0}$
(the vertices of $\mathcal{D}$), and $\left\Vert \cdot\right\Vert _{TV}$
is the total-variation norm:
\begin{equation}
\left\Vert \mu-\nu\right\Vert _{TV}=\max_{A\subseteq\mathcal{D}^{0}}\left|\mu\left(A\right)-\nu\left(A\right)\right|=\tfrac{1}{2}\left\Vert \mu-\nu\right\Vert _{1}.\label{eq:TV}
\end{equation}
A family of digraphs $\left\{ \mathcal{D}_{n}\right\} $ is said to
exhibit cutoff if $\frac{t_{mix}\left(\varepsilon,\mathcal{D}_{n}\right)}{t_{mix}\left(1-\varepsilon,\mathcal{D}_{n}\right)}\overset{{\scriptscriptstyle n\rightarrow\infty}}{\longrightarrow}1$
for every $0<\varepsilon<1$. The cutoff is said to occur at time
$t\left(n\right)$, if for every $\varepsilon>0$ there exists a window
of size $w\left(n,\varepsilon\right)=o\left(t\left(n\right)\right)$,
such that $\left|t_{mix}\left(\varepsilon,\mathcal{D}_{n}\right)-t(n)\right|\leq w(n,\varepsilon)$
for $n$ large enough. If $t\left(n\right)=\log_{k}|\mathcal{D}_{n}|$
we say that the cutoff is \emph{optimal}, since a $k$-regular walk
cannot mix in less steps.

Recall that a connected $k$-regular graph is called a \emph{Ramanujan
graph} if its adjacency spectrum is contained in $\left[-2\sqrt{k-1},2\sqrt{k-1}\right]\cup\left\{ k\right\} $.
\begin{thm*}[\cite{lubetzky2016cutoff}]
The family $\left\{ G_{n}\right\} $ of all $k$-regular Ramanujan
graphs exhibits (1) cutoff for SRW at time $\frac{k}{k-2}\log_{k-1}\left|G_{n}\right|$,
with a window of size $O(\sqrt{\log\left|G_{n}\right|})$; (2) optimal
cutoff for NBRW (at time $\log_{k-1}\left|G_{n}\right|$), with a
window of size $O\left(\log\log\left|G_{n}\right|\right)$.
\end{thm*}
In \cite{lubetzky2016cutoff} SRW-cutoff is first reduced to optimal
NBRW-cutoff, by studying the distance of SRW on the tree from its
starting point. To obtain optimal cutoff for NBRW new spectral techniques
are developed for analyzing non-normal operators.

To see how the notion of Ramanujan graphs generalizes to higher dimension,
recall that $\left[-2\sqrt{k-1},2\sqrt{k-1}\right]$ is the $L^{2}$-adjacency
spectrum of the $k$-regular tree, which is the universal cover of
every $k$-regular graph \cite{Kesten1959}. In accordance, \emph{Ramanujan
complexes }are roughly defined as finite complexes which spectrally
mimic their universal cover; for a precise definition see §\ref{sec:Preliminaries-and-notations}.
In \cite{Lubetzky2017RandomWalks}, a vast generalization of part
(2) of the theorem above is proved: say that a walk rule is \emph{collision-free}
if two walkers which depart from each other can never cross paths
again.
\begin{thm*}[\cite{Lubetzky2017RandomWalks}]
Let $\mathcal{B}$ be an affine Bruhat-Tits building (see §\ref{subsec:Bruhat-Tits-buildings}),
and fix a collision-free walk rule on cells of $\mathcal{B}$. The
family of Ramanujan complexes with universal cover $\mathcal{B}$
exhibit optimal cutoff w.r.t.\ the corresponding walk rule on them.
\end{thm*}
This recovers NBRW on Ramanujan graphs, since NBRW is indeed collision-free
on the edges of the tree. In higher dimension, NBRW is not collision-free
anymore, but in \cite[§5.1]{Lubetzky2017RandomWalks} it is shown
that collision-free walks do exist,\emph{ }on cells of every dimension,
except for vertices. As SRW is not collision-free, the techniques
of \cite{Lubetzky2017RandomWalks} cannot address it (in fact, they
cannot address any operator on vertices - see \cite[Rem.\ 3.5(b)]{Parzanchevski2018RamanujanGraphsDigraphs}).

\medskip{}

The goal of this paper is to establish cutoff for SRW on Ramanujan
complexes. We fix a non-archimedean local field $F$ with residue
field of size $q$, and denote by $\mathcal{B}=\mathcal{B}_{d,F}$
the Bruhat-Tits building associated with $\mathrm{PGL}_{d}\left(F\right)$.
\begin{thm}[Main theorem]
\label{thm:main}The family $\left\{ X_{n}\right\} $ of all Ramanujan
complexes with universal cover $\mathcal{B}$ exhibits total-variation
cutoff for SRW at time $C_{d,q}\log_{q}\left|X_{n}\right|$, with
a window of size $O(\sqrt{\log\left|X_{n}\right|})$. The constant
$C_{d,q}$ is determined in \eqref{eq:Cdq}, and for each $d$ is
a rational function in $q$ of magnitude $\tfrac{1}{\left\lfloor d/2\right\rfloor \left\lceil d/2\right\rceil }+O(\tfrac{1}{q})$
(see Table \ref{tab:deg_xi}).
\end{thm}

We emphasize that the graphs underlying these walks are not Ramanujan
graphs when $d>2$. For example, in the two-dimensional case (where
$d=3$), the $1$-skeleton of $X$ is a $2(q^{2}+q+1)$-regular graph.
If it was a Ramanujan graph, its second largest adjacency eigenvalue
would be bounded by $2\sqrt{2q^{2}+2q+1}\approx2.8q$, but in truth
this eigenvalue equals $6q-o_{n}\left(1\right)$ (cf.\ \cite{saloff1997transition,li2004ramanujan}),
reflecting the abundance of triangles in $X$. In this case, the cutoff
is achieved at time $\frac{q^{2}+q+1}{2(q^{2}-1)}\log_{q}n$ (see
Theorem \ref{thm:cutoff-pgl3}).\medskip{}

A motivation for our result which does not require the notions of
Ramanujan complexes or buildings, is the study of expansion in finite
simple groups (see \cite{Breuillard2018Expansionsimplegroups} for
a recent survey). A celebrated result of Lubotzky, Phillips and Sarnak
\cite{LPS88} uses the building of $\mathrm{PGL}_{2}(\mathbb{Q}_{p})$
to show that the groups $\mathrm{PSL}_{2}(\mathbb{F}_{q})$ have explicit
generators for which the resulting Cayley graphs are Ramanujan, so
that \cite{lubetzky2016cutoff} yields total-variation cutoff for
SRW on these groups. Turning to $\mathrm{PSL}_{d}(\mathbb{F}_{q})$,
the work of \cite{lubetzky2016cutoff} does not apply anymore, since
it is not known whether $\mathrm{PSL}_{d}(\mathbb{F}_{q})$ have generators
which yield Ramanujan Cayley graphs. Nevertheless, by considering
the building of $\mathrm{PGL}_{d}(\mathbb{F}_{q}((t)))$, it was shown
by Lubotzky, Samuels and Vishne (\cite{Lubotzky2005b}, see also \cite{sarveniazi2007explicit})
that the groups $\mathrm{PSL}_{d}(\mathbb{F}_{q})$ have explicit
generators, for which the resulting Cayley graph is precisely the
$1$-skeleton of a Ramanujan complex of type $\widetilde{A}_{d}$.
For $d=3$, such generators can also be given using the building of
$\mathrm{PGL}_{3}(\mathbb{Q}_{p})$ \cite{Evra2018RamanujancomplexesGolden,Ballantine2018ExplicitCayleyRamanujan}.
We thus achieve:
\begin{cor}
\begin{enumerate}
\item Fix $d\geq3$ and a prime power $q$. The family $\{\mathrm{PSL}_{d}(\mathbb{F}_{q^{\ell}})\}_{\ell\rightarrow\infty}$
has an explicit symmetric set of $k=\sum\nolimits _{j=1}^{d-1}\left[\begin{smallmatrix}d\\
j
\end{smallmatrix}\right]_{q}$ \textup{(Gaussian binomial coefficients)} generators exhibiting TV-cutoff.
\item For $d=3$ and infinitely many pairs of primes $p\neq q$, the family
$\mathrm{PSL}_{3}(\mathbb{F}_{q})$ has an explicit symmetric set
of $2\left(p^{2}+p+1\right)$ generators exhibiting TV-cutoff (at
time $\frac{p^{2}+p+1}{2(p^{2}-1)}\log_{p}\left|\mathrm{PSL}_{3}(\mathbb{F}_{q})\right|$).
\end{enumerate}
\end{cor}

We remark that even though this is a claim on Cayley graphs, the proof
makes use of the high-dimensional geometry of their clique complexes!\medskip{}

Let us briefly explain our strategy for proving Theorem \ref{thm:main}.
Given a walk on $X$, we lift it to a walk on $\mathcal{B}$, and
then project it to a sector $\mathcal{S}\subseteq\mathcal{B}$, which
can be identified with the quotient of $\mathcal{B}$ by the stabilizer
of the starting point of the walk. If $X$ is a $k$-regular graph
then $\mathcal{B}$ is the $k$-regular tree, and $\mathcal{S}$ is
simply an infinite ray which we identify with $\mathbb{N}$: SRW on
$\mathcal{B}$ then projects to a $\left(\frac{1}{k},\frac{k-1}{k}\right)$-biased
walk on this ray, and the projected location $\ell\in\mathbb{N}$
of the walker is precisely its distance from the starting point. On
the other hand, this point is also the projection to $\mathcal{S}$
of all terminal vertices of non-backtracking walks of length $\ell$;
combining this with the optimal cutoff for NBRW is used to establish
SRW cutoff in \cite{lubetzky2016cutoff}.

For the building of dimension $d$, the so called \emph{Cartan decomposition
}gives an isomorphism $\mathcal{S}\cong\mathbb{N}^{d}$, and the projected
walk from $\mathcal{B}$ on $\mathcal{S}$ is an explicit, homogeneous,
drifted walk on $\mathbb{N}^{d}$ (with appropriate boundary conditions).
Following the Lubetzky-Peres strategy, we would have liked to use
this walk to reduce SRW-cutoff to some collision-free walk from \cite{Lubetzky2017RandomWalks},
for which optimal cutoff is already established. However, the terminal
vertices of the various walks studied in \cite{Lubetzky2017RandomWalks}
are all located on the special rays in $\mathcal{S}$ which correspond
to the standard axes in $\mathbb{N}^{d}$. This is enough for the
graph case (when $d=1$), but not in general. Our solution combines
all the walks from \cite{Lubetzky2017RandomWalks}, using cells of
all positive dimensions. For each point $\alpha\in\mathcal{S}$, we
construct a concatenation of collision-free walks on cells of different
dimensions, so that the possible paths of the walk terminate in a
uniform vertex in the preimage of $\alpha$ in $\mathcal{B}$. The
results of \cite{Lubetzky2017RandomWalks,Parzanchevski2018RamanujanGraphsDigraphs}
are then used to bound the total-variation mixing of the corresponding
concatenated walk on a Ramanujan complex.

\medskip{}

For the convenience of the reader, we have divided the proof to the
two-dimensional case (namely $\mathrm{PGL}_{3}$) in §\ref{sec:d=00003D3},
and the general case in §\ref{sec:The-d-case}. The case of $d=3$
is considerably simpler, due to the fact that it has additional symmetry:
in this case $\mathrm{PGL}_{3}(F)$ acts transitively on the cells
of every dimension of $\mathcal{B}$ (see \cite{kang2014zeta,Golubev2013triangle,Kaufman2019Freeflagsover}
for detailed combinatorial studies of $\mathcal{B}\left(\mathrm{PGL}_{3}(F)\right)$.)
In addition, it is easier to visualize (see Figure \ref{fig:The-folded-apartment}),
and some computations can be made more explicit and give sharper bounds.
\begin{acknowledgement*}
The authors are grateful to Ori Gurel-Gurevich for his help with proving
Prop.\ \ref{Prop:Transience2}. They also thank Eyal Lubetzky, Alex
Lubotzky and Nati Linial for helpful discussions and encouragement.
M.C.\ was supported by ERC grant 339096 of Nati Linial and by ERC,
BSF and NSF grants of Alex Lubotzky. O.P.\ was supported by ISF grant
2990/21.
\end{acknowledgement*}

\section{\label{sec:Preliminaries-and-notations}Preliminaries and notations}

We briefly recall the notion of Bruhat-Tits buildings of type $\widetilde{A}_{d}$
and the Ramanujan complexes associated with them. For a more detailed
introduction, we refer the reader to \cite{li2004ramanujan,Lubotzky2005a,Lubotzky2013}.

\subsection{\label{subsec:Bruhat-Tits-buildings}Bruhat-Tits buildings}

Let $F$ be a nonarchimedean local field with ring of integers $\mathcal{O}$,
uniformizer $\varpi$, and residue field $\nicefrac{\mathcal{O}}{\varpi\mathcal{O}}$
of size $q$. The simplest examples are $F=\mathbb{Q}_{p}$ with $\left(\mathcal{O},\varpi,q\right)=\left(\mathbb{Z}_{p},p,p\right)$
, and $F=\mathbb{F}_{q}\left(\left(t\right)\right)$ with $\left(\mathcal{O},\varpi,q\right)=\left(\mathbb{F}_{q}\left[\left[t\right]\right],t,q\right)$.
Let $G=\mathrm{PGL}_{d}\left(F\right)$ and $K=\mathrm{PGL}_{d}(\mathcal{O})$,
which is a maximal compact subgroup of $G$. The Bruhat-Tits building
$\mathcal{B}=\mathcal{B}\left(G\right)$ of type $\widetilde{A}_{d-1}$
associated with $G$ is an infinite, contractible, $\left(d-1\right)$-dimensional
simplicial complex, on which $G$ acts faithfully. Denoting by $\mathcal{B}^{j}$
the cells of $\mathcal{B}$ of dimension $j$, the action of $G$
on $\mathcal{B}$ is transitive both on $\mathcal{B}^{0}$ and $\mathcal{B}^{d-1}$.
Furthermore, there is a vertex, which we denote by $\xi$, whose stabilizer
is $K$, so that $\mathcal{B}^{0}$ can be identified with left $K$-cosets
in $G$. In this manner each vertex $g\xi$ is associated with the
$F$-homothety class of the $\mathcal{O}$-lattice $g\mathcal{O}^{d}\leq F^{d}$.
A collection of vertices $\left\{ g_{i}\xi\right\} _{i=0}^{r}$ forms
an $r$-cell if, possibly after reordering, there exist scalars $\alpha_{i}\in F^{\times}$
such that
\[
\varpi g{}_{0}\mathcal{O}^{d}<\alpha_{r}g{}_{r}\mathcal{O}^{d}<\alpha_{r-1}g{}_{r-1}\mathcal{O}^{d}<\ldots<\alpha_{1}g{}_{1}\mathcal{O}^{d}<g{}_{0}\mathcal{O}^{d}.
\]
It follows that the link of a vertex in $\mathcal{B}$ can be identified
with the spherical building of $\mathrm{PGL}_{d}(\mathcal{O}/\varpi\mathcal{O})\cong\mathrm{PGL}_{d}(\mathbb{F}_{q})$,
the finite complex whose cells corresponds to flags in $\mathbb{F}_{q}^{d}$.
In particular, its vertices correspond to nonzero proper subspaces
of $\mathbb{F}_{q}^{d}$, so that the degree of the vertices in $\mathcal{B}$
is 
\[
\deg\left(\xi\right)=\sum\nolimits _{j=1}^{d-1}\left[\begin{smallmatrix}d\\
j
\end{smallmatrix}\right]_{q},\text{ where }\left[\begin{smallmatrix}d\\
j
\end{smallmatrix}\right]_{q}\text{ are Gaussian binomial coefficients}
\]
(see examples in Table \ref{tab:deg_xi}). The vertices of $\mathcal{B}$
are colored by the elements of $\mathbb{Z}/d\mathbb{Z}$, via 
\begin{equation}
\col\left(g\xi\right)=\mathrm{ord}_{\varpi}\det\left(g\right)\in\mathbb{Z}/d\mathbb{Z}\qquad\left(g\in\mathrm{PGL}_{d}(F)\right),\label{eq:color}
\end{equation}
and this coloring makes $\mathcal{B}$ a $d$-partite complex, namely,
every $\left(d-1\right)$-cell contains all colors. For $j\geq1$,
we say that an ordered cell $\sigma\in\mathcal{B}^{j}$ is of \emph{type
one }if $\col\sigma_{i+1}\equiv\col\sigma_{i}+1\Mod{d}$ for $0\leq i<j$,
and we denote by $\mathcal{B}_{1}^{j}$ all $j$-cells of type one.

\subsection{\label{subsec:Ramanujan-complexes}Ramanujan complexes}

A \emph{branching operator }on a set $\Omega$ is a function $T\colon\Omega\to2^{\Omega}.$
By a \emph{geometric }operator $T$ on $\mathcal{B}$ we mean a branching
operator on some subset $\mathcal{C}$ of the cells of $\mathcal{B}$
(e.g., all cells of dimension $j$), which commutes with the action
of $G$. If $\Gamma$ is a torsion-free lattice in $G=\mathrm{PGL}_{d}\left(F\right)$
then the quotient $X=\Gamma\backslash\mathcal{B}$ is a finite complex,
equipped with a covering map $\varphi\colon\mathcal{B}\rightarrow X$,
and $T$ induces a branching operator $T|_{X}$ on the cells $\Gamma\backslash\mathcal{C}$
in $X$, via $T|_{X}=\varphi T\varphi^{-1}$. A function on $X$ is
considered \emph{trivial }if its lift to a function on $\mathcal{B}$
is constant on every orbit of $G'=\mathrm{PSL}_{d}\left(F\right)$,
and an eigenvalue of $T|_{X}$ is called trivial if its eigenfunction
is trivial. Denote by $L_{\col}^{2}\left(X\right)$ the space of trivial
functions, and by $L_{0}^{2}\left(X\right)$ its orthogonal complement.
\begin{defn}
The complex $X=\Gamma\backslash\mathcal{B}$ is called a \emph{Ramanujan
complex }if for every geometric operator $T$ on $\mathcal{C}\subseteq\mathcal{B}$,
the nontrivial spectrum $\mathrm{Spec}(T|_{L_{0}^{2}(\Gamma\backslash\mathcal{C})})$
is contained in the spectrum of $T$ acting on $L^{2}\left(\mathcal{C}\right)$.
\end{defn}

In the above notations, we denote by $\mathcal{D}_{T}\left(\mathcal{B}\right)$
the digraph with vertices $\mathcal{C}$, and edges $\left\{ \sigma\rightarrow\sigma'|\sigma\in\mathcal{C},\sigma'\in T\left(\sigma\right)\right\} $,
and similarly $\mathcal{D}_{T}\left(X\right)$ for the induced digraph
on $\Gamma\backslash\mathcal{C}$.
\begin{thm}[{\cite[Thm.\ 3 and Prop.\ 5.3]{Lubetzky2017RandomWalks}}]
\label{thm:Col-Free-Ram}Let $T$ be a $k$-regular geometric operator
on $\mathcal{B}_{1}^{j}$. If $\mathcal{D}_{T}\left(\mathcal{B}\right)$
is collision-free and $X=\Gamma\backslash\mathcal{B}$ is a Ramanujan
complex, then $\mathcal{D}_{T}\left(X\right)$ is a $\left(d\right)_{j}$-normal
Ramanujan digraph (where $\left(d\right)_{j}=d!/\left(d-j\right)!).$
\end{thm}

This requires some explanation. A $k$-regular digraph $\mathcal{D}$
is called:
\begin{enumerate}
\item \emph{collision-free }if it has at most one directed path between
any two vertices;
\item \emph{$r$-normal} if its adjacency matrix $A_{\mathcal{D}}$ is unitarily
similar to a block diagonal matrix with blocks of size at most $r\times r$;
\item a\emph{ Ramanujan digraph }if the spectrum of $A_{\mathcal{D}}$ is
contained in $\{z\in\mathbb{C}\,|\,\left|z\right|=k\text{ or }\left|z\right|\leq\sqrt{k}\}$.
\end{enumerate}
Denoting by $L_{0}^{2}\left(\mathcal{D}\right)$ the orthogonal complement
to all $A_{\mathcal{D}}$-eigenfunctions with eigenvalue of absolute
value $k$, we have:
\begin{thm}[{\cite[Prop.\ 4.1]{Parzanchevski2018RamanujanGraphsDigraphs}}]
\label{thm:digraph-bound}If $\mathcal{D}$ is a $k$-regular $r$-normal
digraph with $\lambda=\max\{\left|z\right||z\in\mathrm{Spec}(A_{\mathcal{D}}|_{L_{0}^{2}(\mathcal{D})})\}$,
then $\left\Vert \vphantom{\big|}\smash{A_{\mathcal{D}}^{\ell}\big|_{L_{0}^{2}(\mathcal{D})}}\right\Vert _{2}\leq{\ell+r-1 \choose r-1}k^{r-1}\lambda^{\ell-r+1}$.
\end{thm}

In particular, if $\mathcal{D}$ is a $k$-regular $r$-normal Ramanujan
digraph, then we have $\left|\lambda\right|\leq\sqrt{k}$ for every
$\lambda\in\mathrm{Spec}(A_{\mathcal{D}}|_{L_{0}^{2}(\mathcal{D})})$,
so that 
\begin{equation}
\left\Vert A_{\mathcal{D}}^{\ell}\smash{\big|_{L_{0}^{2}\left(\mathcal{D}\right)}}\right\Vert _{2}\leq{\ell+r-1 \choose r-1}k^{(\ell+r-1)/2}\leq\left(\ell+r\right)^{r}k^{(\ell+r)/2}.\label{eq:Ram-dig-power}
\end{equation}

\subsection{\label{subsec:Cartan-decomposition}Cartan decomposition}

With the notations of §\ref{subsec:Bruhat-Tits-buildings}, the \emph{fundamental
apartment} $\mathcal{A}\subseteq\mathcal{B}$ is the subcomplex of
$\mathcal{B}$ induced by all translations of $\xi$ by diagonal matrices
in $G$. Geometrically, $\mathcal{A}$ is a simplicial tessellation
of the affine space $\mathbb{R}^{d-1}$. The edges in $\mathcal{A}$
are as follow: every vertex $\varpi^{\alpha}\xi=\mathrm{diag}\left(\varpi^{\alpha_{1}},...,\varpi^{\alpha_{d-1}},\varpi^{\alpha_{d}}\right)\xi$
is connected to $\varpi^{\alpha+\gamma}\xi$ where $\gamma$ runs
over all non-constant binary vectors, i.e.\ $\gamma\in\left\{ 0,1\right\} ^{d}\backslash\left\{ \zero,\one\right\} $
(here $\zero$ and $\one$ denote the all-zero and all-one vectors
in $\left\{ 0,1\right\} ^{d}$, respectively).

We denote by $\mathcal{S\subseteq\mathcal{A}}$ the sector in $\mathcal{A}$
induced by $A\xi$, where 
\[
A=\{\varpi^{\alpha}=\mathrm{diag}\left(\varpi^{\alpha_{1}},...,\varpi^{\alpha_{d-2}},\varpi^{\alpha_{d-1}},1\right)\,|\,\alpha_{1}\ge...\ge\alpha_{d-1}\ge\alpha_{d}=0\}.
\]
It is easy to see that $\mathcal{S}$ is a fundamental domain for
the action of $S_{d}\leq G$ (the so called \emph{spherical Weyl group})
on $\mathcal{A}$. We identify $\mathcal{S}$ with $\mathbb{N}^{d-1}$
via 
\[
\mathrm{diag}\left(\varpi^{\alpha_{1}},...,\varpi^{\alpha_{d-2}},\varpi^{\alpha_{d-1}},1\right)\xi\mapsto(\alpha_{1}-\alpha_{2},...,\alpha_{d-2}-\alpha_{d-1},\alpha_{d-1}),
\]
thereby giving $\mathbb{N}^{d-1}$ a graph structure. Denote by $\partial\mathcal{S}$
the boundary of $\mathcal{S}$, which corresponds to $\partial\mathbb{N}^{d-1}=\left\{ \vec{x}\in\mathbb{N}^{d-1}\,|\,x_{i}=0\text{ for some }i\right\} $.
Except at $\partial\mathcal{S}$, the edges are the same as in $\mathcal{A}$,
parameterized by $\gamma\in\left\{ 0,1\right\} ^{d}\backslash\left\{ \zero,\one\right\} $.
For $\varpi^{\alpha}\in\partial\mathcal{S}$, it might happen that
$\varpi^{\alpha+\gamma}\notin\mathcal{S}$, e.g.\ when $\alpha_{i}+\gamma_{i}>\alpha_{i-1}+\gamma_{i-1}$,
and one obtains the appropriate terminus of $\gamma$ by reordering
the entries of $\varpi^{\alpha+\gamma}$ in descending order, and
then dividing it by its last coordinate if it is not $1$. The case
of $d=3$ is depicted in Figure \ref{fig:The-folded-apartment}.

The Cartan decomposition for $\mathrm{PGL}_{d}$ states that
\[
G=\bigsqcup\nolimits _{a\in A}KaK,\quad\text{ or (equivalently) }\quad\mathcal{B}^{0}=\bigsqcup\nolimits _{a\in A}Ka\xi,
\]
and the proof is a simple exercise (see e.g.\ \cite[§13.2]{Goldfeld2011AutomorphicrepresentationsL}).
It follows that $\mathcal{S}$ can also be identified with the quotient
of $\mathcal{B}$ by $K$, and we denote the obtained projection from
$\mathcal{B}$ to $\mathbb{N}^{d-1}$ by $\Phi$. In conclusion, we
have identified four complexes: $K\backslash\mathcal{B}\cong S_{d}\backslash\mathcal{A}\cong\mathcal{S}\cong\mathbb{N}^{d-1}$.

\section{The ${\rm PGL}_{3}$ case\label{sec:d=00003D3}}

In this section $\BT=\mathcal{B}_{3,F}$ is the two-dimensional Bruhat-Tits
building of $G=\mathrm{PGL}_{3}\left(F\right)$. The $1$-skeleton
of $\BT$ is a $k$-regular graph, with $k=\deg\left(\xi\right)=2(q^{2}+q+1)$,
where $q$ is the size of the residue field of $F$.
\begin{thm}
\label{thm:cutoff-pgl3}Let $X=\Gamma\backslash\mathcal{B}_{3,F}$
be a Ramanujan complex with $n$ vertices. Then SRW on the underlying
graph of $X$ has total-variation cutoff at time $\frac{q^{2}+q+1}{q^{2}-1}\log_{q^{2}}n$
with a window of size $O(\sqrt{\log n})$.
\end{thm}

\subsection{\label{subsec:A-lower-bound}A lower bound on the mixing time}

Throughout this section we fix $\varepsilon>0$. Denote by $B(\xi,r)$
the $r$-ball around $\xi$, i.e.\ the vertices of graph distance
at most $r$ from $\xi$ in $\mathcal{B}$. First, we show that the
ball of radius 
\[
r_{0}=\log_{q^{2}}n-3\log_{q^{2}}\log_{q^{2}}n
\]
can cover only a small fraction of any $n$-vertex quotient of $\mathcal{B}$:
\begin{prop}
\label{prop:r0-ball-PGL3}For $n$ large enough, $\left|B(\xi,r_{0})\right|\leq\varepsilon n$.
\end{prop}

\begin{proof}
Given $r\geq1$, the $r$-sphere $S\left(\xi,r\right)$ is shown in
\cite{Evra2018RamanujancomplexesGolden} to be of size
\[
\left|S\left(\xi,r\right)\right|=(r+1)q^{2r}+2rq^{2r-1}+2rq^{2r-2}+(r-1)q^{2r-3}.
\]
Thus, one can crudely bound the size of the $r$-ball by $|B(\xi,r)|\leq8r^{2}q^{2r}$,
hence 
\[
|B(\xi,r_{0})|\leq8r_{0}^{2}q^{2r_{0}}\leq\frac{8(\log_{q^{2}}n)^{2}}{(\log_{q^{2}}n)^{3}}n\leq\varepsilon n
\]
for $n$ large enough.
\end{proof}
Let $(\X_{t})$ be a SRW on $\BT$ starting at $\xi$. We would like
to determine until when does the walk remains in the $r_{0}$-ball
around $\xi$ with high probability. Since the distance from $\xi$
is $K$-invariant, we have $\dist\left(\zeta,\xi\right)=\dist(\Phi\left(\zeta\right),\Phi\left(\xi\right))=\dist(\Phi\left(\zeta\right),(0,0))$
for $\zeta\in\mathcal{B}^{0}$, which leads us to consider the projection
of $\mathcal{X}_{t}$ by $\Phi$. In this manner, we obtain a (non-simple)
random walk $(\Phi(\mathcal{X}_{t}))$ on $\mathbb{N}^{2}$ , and
we define
\[
\rho\left(t\right)=\dist\left(\Phi\left(\mathcal{X}_{t}\right),(0,0)\right)=\dist\left(\mathcal{X}_{t},\xi\right).
\]
Recall that we identified $\mathcal{S}\cong\mathbb{N}^{2}$ by mapping
$\mathrm{diag}\left(\varpi^{\alpha},\varpi^{\beta},1\right)\xi$ to
$\left(\alpha-\beta,\beta\right)$, and the edges in $\mathbb{N}^{2}$
(except at the boundary) are $\pm(1,0),\pm(0,1),\pm(1,-1)$ - see
Figure~\ref{fig:The-folded-apartment}.

\begin{figure}[h]
\begin{centering}
\includegraphics[scale=0.6]{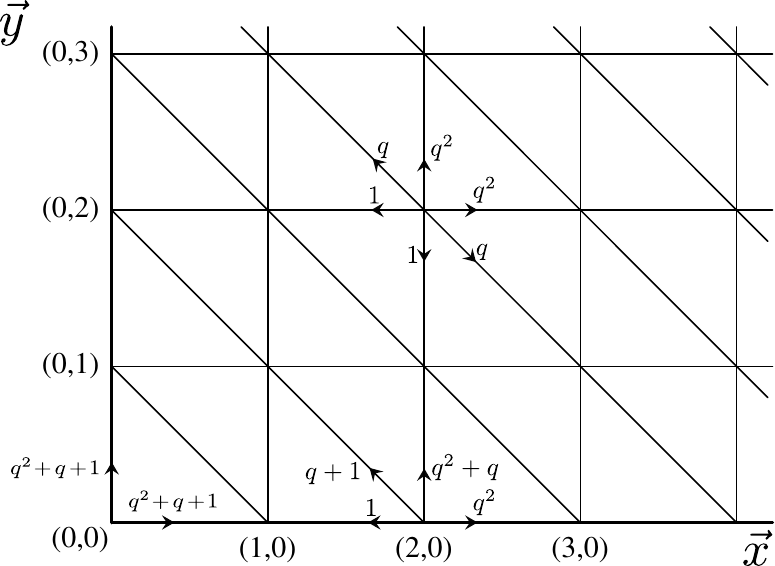}
\par\end{centering}
\caption{\label{fig:The-folded-apartment}The sector $\mathcal{S}\subseteq\mathcal{B}$
as $\mathbb{N}^{2}$, and transition probabilities projected from
SRW on $\mathcal{B}$, scaled by $2\left(q^{2}+q+1\right)$.}
\end{figure}

Let $\vec{x}$ and $\vec{y}$ be the boundary lines of $\mathcal{S}$
(the $x$ and $y$ axes in Figure \ref{fig:The-folded-apartment}).
The transition probabilities of the projected random walk are as follows:
from $\left(0,0\right)$ there is a probability of $\frac{1}{2}$
of moving to $\left(1,0\right)$ and to $\left(0,1\right)$. Outside
the boundary, the edge $\left(\Delta_{x},\Delta_{y}\right)$ is taken
with probability $q^{\Delta_{x}+\Delta_{y}+1}/k$. On $\vec{x}\backslash(0,0)$,
the edges with $\Delta_{y}=-1$ are folded back in, giving the probabilities
shown in Figure \ref{fig:The-folded-apartment}, and on $\vec{y}$
the folding is symmetric.

Denote $y(t)=\dist(\Phi\left(\X_{t}\right),\vec{x})$ and $x(t)=\dist(\Phi\left(\X_{t}\right),\vec{y})$,
which measure the distance of the projected walk from the boundary.
Clearly, $\rho(t)=y(t)+x(t)$. We consider $y(t)$, $x(t)$ and $\rho\left(t\right)$
as random walks on $\mathbb{N}$ starting at zero.
\begin{prop}
\label{prop:PGL3_transient}The walks $x(t)$ and $y(t)$ are transient.
\end{prop}

\begin{proof}
We treat only $x\left(t\right)$, and the proof for $y\left(t\right)$
is analogous. Although the distribution of the random variable $\partial x(t)=x(t)-x(t-1)$
depends on the position of the walk at time $t-1$, there are only
four cases to consider: when the walk is at the origin, when the walk
is on $\vec{x}$ or $\vec{y}$, and when both $x(t-1)$ and $y(t-1)$
are positive. In all of these cases, 
\[
\mathbb{E}[\partial x(t)]\geq\frac{q^{2}-q-2}{2(q^{2}+q+1)}.
\]
Thus, for $q>2$, the value of $x(t)$ is expected to strictly grow
at each step and thus $x(t)$ is transient. To cover the case of $q=2$,
one can look ``two steps ahead'', namely on $\partial^{2}x(t)=x(t)-x(t-2)$.
There are more cases to check, but explicit computation shows that
\[
\mathbb{E}\left[\partial^{2}x(t)\right]\geq\tfrac{4q^{4}+q^{3}-5q^{2}-9q-7}{4(q^{2}+q+1)^{2}},
\]
which is positive for all $q\geq2$, giving again transience.
\end{proof}
Define $S(t)=\sum_{i=1}^{t}Y_{i}$, where $Y_{i}=\rho(i)-\rho(i-1)$
whenever $x(i-1),y(i-1)>0$, and otherwise $Y_{i}$ is a random variable
independent of any other, attaining $1,0,-1$ with respective probabilities
$2q^{2}/k$, $2q/k$, $2/k$. It follows that the $Y_{i}$'s are i.i.d.,
and by the central limit theorem, 
\begin{equation}
\widetilde{S}(t)=\frac{S(t)-\mathscr{E}t}{\sigma\sqrt{t}}\Rightarrow\mathcal{N}(0,1),\text{ where }\left\{ \ensuremath{\begin{alignedat}{1}\mathscr{E} & =\tfrac{q^{2}-1}{q^{2}+q+1}\\
\sigma & =\tfrac{\sqrt{q^{3}+4q^{2}+q}}{q^{2}+q+1}
\end{alignedat}
}\right.\label{eq:E-s-def}
\end{equation}
and $\mathcal{N}(0,1)$ is the standard normal distribution. Let $\mathcal{E}(t)=\frac{\rho(t)-S(t)}{\sigma\sqrt{t}}$.
Since $x(i)$ and $y(i)$ are transient, the difference $\rho(t)-S(t)$
is bounded with probability $1$, so that $\mathbb{P}\left[|\mathcal{E}(t)|<C\right]\xrightarrow{t\to\infty}1$
for every $C>0$. Hence $\mathcal{E}\left(t\right)$ converges to
the Dirac measure concentrated at $0$, and
\begin{equation}
\Xi(t)=\frac{\rho(t)-\mathscr{E}t}{\sigma\sqrt{t}}=\widetilde{S}(t)+\mathcal{E}(t)\Rightarrow\mathcal{N}(0,1).\label{eq:Xi-dist}
\end{equation}
Recall that $\varepsilon>0$ and $r_{0}$ were fixed at the beginning
of section \ref{subsec:A-lower-bound}.
\begin{prop}
\label{prop:t0-to-r0}For $n$ large enough and any $s\geq0$, at
time 
\begin{equation}
t_{0}=t_{0}\left(s\right)=\frac{q^{2}+q+1}{q^{2}-1}\log_{q^{2}}n-(s+1)\sqrt{\log_{q^{2}}n}\label{eq:TTT}
\end{equation}
the distance of $\mathcal{X}_{t}$ from $\xi$ satisfies
\[
\mathbb{P}\left[\rho(t_{0})>r_{0}\right]<\mathbb{P}\left[Z>c_{q}\cdot s\right]+\varepsilon,
\]
where $Z\sim\mathcal{N}(0,1)$ and $c_{q}=\mathscr{E}^{3/2}/\sigma=\sqrt{\frac{(q^{2}-1)^{3}}{(q^{2}+q+1)(q^{3}+4q^{2}+q)}}$
(see \eqref{eq:E-s-def}).
\end{prop}

\begin{proof}
We note that $\rho(t_{0})>r_{0}$ is equivalent to 
\[
\Xi(t_{0})>\frac{r_{0}-\mathscr{E}t_{0}}{\sigma\sqrt{t_{0}}}=\frac{(s+1)\mathscr{E}\sqrt{\log_{q^{2}}n}-3\log_{q^{2}}\log_{q^{2}}n}{\sigma\sqrt{t_{0}}}.
\]
For $n$ large enough we have $\mathscr{E}\sqrt{\log_{q^{2}}n}\geq3\log_{q^{2}}\log_{q^{2}}n$,
and thus
\[
\frac{(s+1)\mathscr{E}\sqrt{\log_{q^{2}}n}-3\log_{q^{2}}\log_{q^{2}}n}{\sigma\sqrt{t_{0}}}>\frac{s\mathscr{E}\sqrt{\log_{q^{2}}n}}{\sigma\sqrt{t_{0}}},
\]
and from $\sqrt{t_{0}}<\mathscr{E}^{-1/2}\sqrt{\log_{q^{2}}n}$ follows
that
\[
\frac{s\mathscr{E}\sqrt{\log_{q^{2}}n}}{\sigma\sqrt{t_{0}}}>\frac{s\mathscr{E}^{3/2}}{\sigma}=c_{q}s.
\]
Lastly, since $\Xi\left(t_{0}\right)$ converges to $Z$ in distribution
and $t_{0}\overset{{\scriptscriptstyle n\rightarrow\infty}}{\longrightarrow}\infty$,
for $n$ large enough we have
\[
\left|\mathbb{P}\left[\Xi(t_{0})>c_{q}s\right]-\mathbb{P}\left[Z>c_{q}s\right]\right|<\eps.
\]
All in all we conclude that
\[
\mathbb{P}\left[\rho(t_{0})>r_{0}\right]\leq\mathbb{P}\left[\Xi(t_{0})>c_{q}s\right]\leq\mathbb{P}\left[Z>c_{q}s\right]+\varepsilon.\qedhere
\]
\end{proof}
Now let $X$ be a quotient of $\mathcal{B}$ with $n$ vertices. For
any $v\in X^{0}$ we can choose the covering map $\varphi\colon\mathcal{B}\to X$
to satisfy $\varphi\left(\xi\right)=v$. This map induces a correspondence
between paths in $X$ starting at $v$ and paths in $\BT$ starting
at $\xi$, and in particular, $\varphi\left(B\left(\xi,r\right)\right)=B\left(v,r\right)$.
The projection $X_{t}=\varphi\left(\mathcal{X}_{t}\right)$ is a SRW
on (the $1$-skeleton of) $X$ starting at $v$. We recall that $\mu_{X}^{t}=\mu_{X,v}^{t}$
denotes the distribution of $(X_{t})$ and $\pi_{X}$ the uniform
distribution on $X^{0}$.
\begin{prop}
\label{prop:mix-t0}There exists $s=s\left(q,\varepsilon\right)$
such that for $n$ large enough, the $\left(1\!-\!3\varepsilon\right)$-mixing
time of SRW on $X$ is at least $t_{0}=t_{0}\left(s\right)$.
\end{prop}

\begin{proof}
Using $\varphi\left(B\left(\xi,r\right)\right)=B\left(v,r\right)$,
which implies in particular $\mu_{X,v}^{t}\left(B\left(v,r\right)\right)\geq\mu_{\mathcal{B},\xi}^{t}\left(B\left(\xi,r\right)\right)$,
together with Prop.\ \ref{prop:t0-to-r0} and Prop.\ \ref{prop:r0-ball-PGL3},
we have for $n$ large enough
\begin{align*}
\big\Vert\mu_{X,v}^{t_{0}}-\pi_{X}\big\Vert_{TV} & \geq\pi_{X}\left(X^{0}\backslash B(v,r_{0})\right)-\mu_{X,v}^{t_{0}}\left(X^{0}\backslash B(v,r_{0})\right)\\
 & \geq\tfrac{n-\left|B(v,r_{0})\right|}{n}-\mu_{\mathcal{B},\xi}^{t_{0}}\left(\mathcal{B}^{0}\backslash B(\xi,r_{0})\right)\\
 & \geq\tfrac{n-\left|B(\xi,r_{0})\right|}{n}-\mathbb{P}\left[Z>c_{q}s\right]-\varepsilon\\
 & \geq1-2\varepsilon-\mathbb{P}\left[Z>c_{q}s\right].
\end{align*}
This implies in particular $\max_{v\in X^{0}}||\mu_{X,v}^{t_{0}}-\pi_{X}||_{TV}\geq1-2\varepsilon-\mathbb{P}\left[Z>c_{q}s\right]$,
and we can choose $s$ such that $\mathbb{P}[Z>c_{q}s]<\varepsilon$,
and thus $t_{mix}(1-3\eps)>t_{0}$.
\end{proof}

\subsection{An upper bound for the mixing time}

Recall from \eqref{eq:color} that $\mathcal{B}$ is tri-partite via
$\col:\mathcal{B}^{0}\rightarrow\mathbb{Z}/3\mathbb{Z}$. The quotient
$X=\Gamma\backslash\mathcal{B}$ is tripartite if and only if the
map $\col$ factors through $X^{0}$, which is equivalent to $\mathrm{ord}_{\varpi}\det\gamma\in3\mathbb{Z}$
for all $\gamma\in\Gamma$. When this is the case, the trivial functions
$L_{\col}^{2}(X^{0})$ (see §\ref{subsec:Ramanujan-complexes}) are
those which are constant on each color, and when $X$ is not tri-partite,
$L_{\col}^{2}(X^{0})$ are the constant functions. Denote by $\mathcal{P}_{\!\col}$
and $\mathcal{P}_{0}$ the orthogonal projections corresponding to
the decomposition $L^{2}(X^{0})=L_{\col}^{2}(X^{0})\oplus L_{0}^{2}(X^{0})$.
For any $t$, we have
\begin{equation}
\left\Vert \mu_{X}^{t}-\pi_{X}\right\Vert _{TV}\leq\left\Vert \mathcal{P}_{0}(\mu_{X}^{t})\right\Vert _{TV}+\left\Vert \mathcal{P}_{\!\col}(\mu_{X}^{t})-\pi_{X}\right\Vert _{TV}.\label{eq:P0_Pcol}
\end{equation}
We first bound the second term:
\begin{prop}
\label{prop:triangle-mix}There exist $t_{\triangle}=t_{\triangle}(\varepsilon)$
such that $\left\Vert \mathcal{P}_{\!\col}(\mu_{X}^{t})-\pi_{X}\right\Vert _{TV}\leq\varepsilon$
for any $t\geq t_{\triangle}$.
\end{prop}

\begin{proof}
If $X$ is non-tripartite then $\mathcal{P}_{\!\col}(\mu_{X}^{t})=\pi_{X}$,
as both are constant functions of sum one. If $X$ is tripartite,
$\col$ induces a simplicial map $\col:X\rightarrow\triangle$, where
$\triangle$ is the 2-simplex with vertices $\mathbb{Z}/3\mathbb{Z}$.
In this case, $\mathcal{P}_{\!\col}(\mu_{X}^{t})$ is the pullback
of SRW on the 2-simplex starting at $0=\col\left(v\right)$, i.e.,
$\mathcal{P}_{\!\col}(\mu_{X}^{t})\left(w\right)=\frac{3}{n}\cdot\mu_{\triangle}^{t}\left(\col\left(w\right)\right)$.
The triangle is connected and non-bipartite, so there exists a time
$t_{\triangle}$, not depending on $n$, such that $\Vert\mu_{\triangle}^{t}-\pi_{\Delta}\Vert_{TV}<\varepsilon$
for $t>t_{\triangle}$, hence $\left\Vert \smash{\mathcal{P}_{\!\col}(\mu_{X}^{t})}-\pi_{X}\right\Vert _{TV}=\Vert\mu_{\Delta}^{t_{1}}-\pi_{\Delta}\Vert_{TV}\leq\varepsilon$.
\end{proof}
Next, we define 
\begin{align*}
r_{1} & =\log_{q^{2}}n+16\log_{q^{2}}\log_{q^{2}}n\\
t_{1} & =\frac{q^{2}+q+1}{q^{2}-1}\log_{q^{2}}n+(s+1)\sqrt{\log_{q^{2}}n},
\end{align*}
where $s$ is as in Prop.\ \ref{prop:mix-t0}. Observe that by time
$t_{1}$ SRW on $\mathcal{B}$ leaves $B(\xi,r_{1})$ with high probability:
the same arguments as in Props.\ \ref{prop:t0-to-r0} and \ref{prop:mix-t0}
give for $n$ large enough 
\begin{equation}
\mathbb{P}\left[\rho(t_{1})<r_{1}\right]\leq\mathbb{P}[Z>c_{q}s]+\varepsilon<2\varepsilon.\label{eq:rho_at_time_t1}
\end{equation}
It is left to bound $\left\Vert \mathcal{P}_{0}(\mu_{X}^{t_{1}})\right\Vert _{TV}$,
and for this we use for the first time the assumption that $X$ is
a Ramanujan complex. We decompose $\mu_{X}^{t_{1}}$ by conditioning
on the values of $\rho,x,y$ at time $t_{1}$: denoting $\mu_{X}^{t,x,y}=\mathbb{P}\big[X_{t}=\cdot\,\big|\,\begin{smallmatrix}x(t)=x,\\
y(t)=y\phantom{,}
\end{smallmatrix}\big]$, we have
\begin{align}
\left\Vert \mathcal{P}_{0}(\mu_{X}^{t_{1}})\right\Vert _{TV} & =\Big\Vert\mathbb{P}\left[\rho(t_{1})\!<\!r_{1}\right]\mathcal{P}_{0}\left(\mathbb{P}[X_{t}\negmedspace=\cdot\,\middle|\,\rho(t_{1})\!<\!r_{1}]\right)\nonumber \\
 & \hphantom{=\Big\Vert}+\sum\nolimits _{r_{1}\leq x+y}\mathbb{P}\big[\begin{smallmatrix}x(t_{1})=x\\
y(t_{1})=y
\end{smallmatrix}\big]\mathcal{P}_{0}(\mu_{X}^{t_{1},x,y})\Big\Vert_{TV}\label{eq:P0_bound}\\
 & \leq2\varepsilon+\max_{r_{1}\leq x+y\leq t_{1}}\big\Vert\mathcal{P}_{0}(\mu_{X}^{t_{1},x,y})\big\Vert_{TV},\nonumber 
\end{align}
using $x\left(t\right)+y\left(t\right)\leq t_{1}$ and \eqref{eq:rho_at_time_t1}.
To understand the $L_{0}^{2}$-projection of the conditional distribution
$\mu_{X}^{t_{1},x,y}$, we turn to study the fiber $\Phi^{-1}\left(x,y\right)$,
using carefully chosen geometric operators on the cells of $\mathcal{B}$.

Recall the definition of cells of type one from §\ref{subsec:Bruhat-Tits-buildings}.
While $g\in G$ does not preserve colors in $\mathcal{B}^{0}$ in
general, it does preserve the difference between colors, so that the
cells of type one in $X$ are well defined (namely, $X_{1}^{j}=\Gamma\backslash\mathcal{B}_{1}^{j}$).
Let $T_{1}$ and $T_{2}$ be the geodesic edge-flow and triangle-flow
operators from \cite[§5.1]{Lubetzky2017RandomWalks}: the operator
$T_{1}$ acts on $\mathcal{B}_{1}^{1}$, taking a (directed) edge
$vw$ to all edges $wu$ of type one such that $vwu$ is not a triangle
in $\mathcal{B}$. The operator $T_{2}$ acts on $\mathcal{B}_{1}^{2}$,
taking the (ordered) triangle $vwu$ to all triangles $wuy$ with
$y\neq v$. We introduce the operators:
\begin{align*}
T_{01} & :\mathcal{B}^{0}\rightarrow\mathcal{B}_{1}^{1} & T_{01}\left(v\right) & =\left\{ wv\,\middle|\,w\in\mathcal{B}^{0}\quad\text{(and \ensuremath{wv} is of type one)}\right\} \\
T_{12} & :\mathcal{B}_{1}^{1}\rightarrow\mathcal{B}_{1}^{2} & T_{12}\left(wv\right) & =\left\{ uwv\,\middle|\,uwv\in\mathcal{B}_{1}^{2}\right\} \\
T_{20} & :\mathcal{B}_{1}^{2}\rightarrow\mathcal{B}^{0} & T_{20}\left(uwv\right) & =\left\{ v\right\} .
\end{align*}
 All of the operators $T_{i},T_{ij}$ are regular and geometric.
\begin{prop}
\label{prop:fiber-by-walks}For any $\left(x,y\right)\in\mathbb{N}^{2}$,
we have $\Phi^{-1}\left(x,y\right)=T_{\left(x,y\right)}\left(\xi\right)$,
where 
\[
T_{\left(x,y\right)}=T_{20}\circ T_{2}^{2y}\circ T_{12}\circ T_{1}^{x}\circ T_{01}:\mathcal{B}^{0}\rightarrow\mathcal{B}^{0}.
\]
\end{prop}

\begin{proof}
If $\sigma_{1},\sigma_{2}$ are two cells in $\mathcal{B}$ with corresponding
$G$-stabilizers $G_{\sigma_{i}}$, any double coset $G_{\sigma_{1}}gG_{\sigma_{2}}$
defines a geometric branching operator from the orbit $G\sigma_{1}$
to $G\sigma_{2}$, by 
\begin{equation}
\left(G_{\sigma_{1}}gG_{\sigma_{2}}\right)\left(g'\sigma_{1}\right)=g'G_{\sigma_{1}}g\sigma_{2}.\label{eq:Hecke-op}
\end{equation}
Defining $e_{1}=\mathrm{diag}(\varpi,\varpi,1)\xi\rightarrow\xi$
and $\tau_{1}=\left[\mathrm{diag}(\varpi,1,1)\xi,\mathrm{diag}(\varpi,\varpi,1)\xi,\xi\right]$
we have orbits $\mathcal{B}^{0}=G\xi$, $\mathcal{B}_{1}^{1}=Ge_{1}$,
$\mathcal{B}_{1}^{2}=G\tau_{1}$, and stabilizers
\[
K=G_{\xi},\quad P_{1}=G_{e_{1}}=\left(\begin{smallmatrix}\mathcal{O} & \mathcal{O} & \mathcal{O}\\
\varpi\mathcal{O} & \mathcal{O} & \mathcal{O}\\
\varpi\mathcal{O} & \mathcal{O} & \mathcal{O}
\end{smallmatrix}\right)\cap K,\quad P_{2}=G_{\tau_{1}}=\left(\begin{smallmatrix}\mathcal{O} & \mathcal{O} & \mathcal{O}\\
\varpi\mathcal{O} & \mathcal{O} & \mathcal{O}\\
\varpi\mathcal{O} & \varpi\mathcal{O} & \mathcal{O}
\end{smallmatrix}\right)\cap K.
\]
The operators we defined arise as $T_{01}=KP_{1}$, $T_{1}=P_{1}\left(\begin{smallmatrix}\varpi\\
 & 1\\
 &  & 1
\end{smallmatrix}\right)P_{1}$, $T_{12}=P_{1}P_{2}$, $T_{2}=P_{2}\left(\begin{smallmatrix} & 1\\
\varpi\\
 &  & 1
\end{smallmatrix}\right)P_{2}$, and $T_{20}=P_{2}K$. Thus, successively applying \eqref{eq:Hecke-op}
we obtain 
\begin{align*}
T_{\left(x,y\right)}\left(\xi\right) & =KP_{1}\left(P_{1}\left(\begin{smallmatrix}\varpi\\
 & 1\\
 &  & 1
\end{smallmatrix}\right)P_{1}\right)^{x}P_{1}P_{2}\left(P_{2}\left(\begin{smallmatrix} & 1\\
\varpi\\
 &  & 1
\end{smallmatrix}\right)P_{2}\right)^{2y}P_{2}\xi.
\end{align*}
Explicit computation in \cite[§5.1]{Lubetzky2017RandomWalks} shows
that 
\[
\left(P_{1}\left(\begin{smallmatrix}\varpi\\
 & 1\\
 &  & 1
\end{smallmatrix}\right)P_{1}\right)^{x}=\left\{ \left(\begin{smallmatrix}\varpi^{x} & \alpha & \beta\\
 & 1\\
 &  & 1
\end{smallmatrix}\right)\,\middle|\,\alpha,\beta\in\mathcal{O}/\varpi^{x}\mathcal{O}\right\} P_{1},
\]
and we note that $K\left(\begin{smallmatrix}\varpi^{x} & \alpha & \beta\\
 & 1\\
 &  & 1
\end{smallmatrix}\right)=K\left(\begin{smallmatrix}\varpi^{x}\\
 & 1\\
 &  & 1
\end{smallmatrix}\right)$ (for $\alpha,\beta\in\mathcal{O}$), so that 
\[
T_{\left(x,y\right)}\left(\xi\right)=K\left(\begin{smallmatrix}\varpi^{x}\\
 & 1\\
 &  & 1
\end{smallmatrix}\right)P_{1}P_{2}\left(P_{2}\left(\begin{smallmatrix} & 1\\
\varpi\\
 &  & 1
\end{smallmatrix}\right)P_{2}\right)^{2y}\xi
\]
(we have used also $P_{1},P_{2}\leq K$). Denoting $K_{1,2}=\big\{\left(\begin{smallmatrix}\mu\\
 & A
\end{smallmatrix}\right)\,|\,\mu\in\mathcal{O^{\times}},A\in GL_{2}\left(\mathcal{O}\right)\big\}$, one can verify that $P_{1}P_{2}\subseteq K_{1,2}P_{2}$ (in fact,
$P_{1}P_{2}=\left\{ I,\left(\begin{smallmatrix}1\\
 & \mathcal{O} & 1\\
 & 1
\end{smallmatrix}\right)\right\} P_{2}$), and since the elements of $K_{1,2}$ commute with $\left(\begin{smallmatrix}\varpi^{x}\\
 & 1\\
 &  & 1
\end{smallmatrix}\right)$ this implies 
\[
T_{\left(x,y\right)}\left(\xi\right)=K\left(\begin{smallmatrix}\varpi^{x}\\
 & 1\\
 &  & 1
\end{smallmatrix}\right)\left(P_{2}\left(\begin{smallmatrix} & 1\\
\varpi\\
 &  & 1
\end{smallmatrix}\right)P_{2}\right)^{2y}\xi.
\]
Finally, explicit computation shows that
\[
\left(P_{2}\left(\begin{smallmatrix} & 1\\
\varpi\\
 &  & 1
\end{smallmatrix}\right)P_{2}\right)^{2y}=\left\{ \left(\begin{smallmatrix}\varpi^{y} &  & \alpha\\
 & \varpi^{y} & \beta\\
 &  & 1
\end{smallmatrix}\right)\,\middle|\,\alpha,\beta\in\mathcal{O}/\varpi^{y}\mathcal{O}\right\} P_{2},
\]
yielding (with $\alpha,\beta$ ranging over $\mathcal{O}/\varpi^{y}\mathcal{O}$)
\[
T_{\left(x,y\right)}\left(\xi\right)=K\left(\begin{smallmatrix}\varpi^{x+y} &  & \varpi^{x}\alpha\\
 & \varpi^{y} & \beta\\
 &  & 1
\end{smallmatrix}\right)\xi=K\left(\begin{smallmatrix}\varpi^{x+y}\\
 & \varpi^{y}\\
 &  & 1
\end{smallmatrix}\right)\xi=\Phi^{-1}\left(x,y\right).\qedhere
\]
\end{proof}
\begin{prop}
\label{prop:P0_xy_bound}If $r_{1}\leq x+y\leq t_{1}$, then for $n$
large enough $\big\Vert\mathcal{P}_{0}(\mu_{X}^{t_{1},x,y})\big\Vert_{TV}\leq\varepsilon$.
\end{prop}

\begin{proof}
Recall that $\varphi$ induces a correspondence between SRW on $\mathcal{B}$
and $X$, so that $\mu_{X}^{t,x,y}$ (for any $t,x,y$) is the pushforward
of $\mu_{\mathcal{B},\xi}^{t,x,y}$ by $\varphi$: 
\[
\mu_{X}^{t,x,y}(w)=\mu_{\mathcal{B},\xi}^{t,x,y}(\varphi^{-1}\left(w\right))=\mathbb{P}\big[\mathcal{X}_{t}\in\varphi^{-1}\left(w\right)\,\big|\,\begin{smallmatrix}x(t)=x\\
y(t)=y
\end{smallmatrix}\big]\qquad\forall w\in X^{0}.
\]
It follows from the Cartan decomposition that the distances from $\vec{x}$
and $\vec{y}$ together determine a unique $K$-orbit in $\mathcal{B}^{0}$.
Since the SRW on $\mathcal{B}$ commutes with $K$, this implies that
for any distance profile $(x,y)\in\mathbb{N}\times\mathbb{N}$ the
distribution $\mu_{\mathcal{B},\xi}^{t,x,y}$ is the uniform distribution
over $\Phi^{-1}\left(x,y\right)$, which we denote by $\pi_{x,y}$.
We conclude that $\mu_{X}^{t,x,y}=\pi_{x,y}\circ\varphi^{-1}$. For
any of the geometric operators $T=T_{i},T_{ij},T_{\left(x,y\right)}$,
we denote by $\widetilde{T}$ the corresponding stochastic operator
on $L^{2}$-spaces, e.g.
\[
\widetilde{T}_{01}:L^{2}\left(\mathcal{B}^{0}\right)\rightarrow L^{2}\left(\mathcal{B}_{1}^{1}\right),\qquad\big(\widetilde{T}_{01}f\big)\left(e\right)=\sum_{w:e\in T_{01}\left(w\right)}\frac{f\left(w\right)}{|T_{01}\left(w\right)|}.
\]
By Prop.\ \ref{prop:fiber-by-walks}, $\supp\widetilde{T_{\left(x,y\right)}}\left(\one_{\xi}\right)\subseteq\Phi^{-1}\left(x,y\right)$.
Furthermore, $\widetilde{T_{\left(x,y\right)}}\left(\one_{\xi}\right)$
is $K$-invariant as $\widetilde{T_{\left(x,y\right)}}\left(\one_{\xi}\right)\left(k\xi'\right)=\widetilde{T_{\left(x,y\right)}}\left(\one_{k^{-1}\xi}\right)\left(\xi'\right)=\widetilde{T_{\left(x,y\right)}}\left(\one_{\xi}\right)\left(\xi'\right)$,
hence $\widetilde{T_{\left(x,y\right)}}\left(\one_{\xi}\right)=\pi_{x,y}$.
The stochastic operator $\widetilde{T}|_{X}$ corresponding to $T|_{X}=\varphi T\varphi^{-1}$
satisfies $(\widetilde{T}\mu)\!\circ\!\varphi^{-1}=\widetilde{T}|_{X}(\mu\!\circ\!\varphi^{-1})$
for any distribution $\mu$ on $\mathcal{B}$, so that
\begin{equation}
\mu_{X}^{t,x,y}=\widetilde{T_{\left(x,y\right)}}\left(\one_{\xi}\right)\circ\varphi^{-1}=\widetilde{T_{\left(x,y\right)}}\big|_{X}\left(\one_{v}\right)=\widetilde{T_{20}}\widetilde{T_{2}^{2y}}\widetilde{T_{12}}\widetilde{T_{1}^{x}}\widetilde{T_{01}}\big|_{X}\left(\one_{v}\right).\label{eq:L2-decomposition}
\end{equation}
It follows from the regularity of incidence relations in $X$ that
the operators $\widetilde{T_{i}}|_{X}$ and $\widetilde{T_{ij}}|_{X}$
decompose with respect to the direct sums $L^{2}=L_{\col}^{2}\oplus L_{0}^{2}$
of the appropriate cells, and in particular $\mathcal{P}_{0}(\mu_{X}^{t,x,y})=\widetilde{T_{\left(x,y\right)}}\big|_{X}\left(\mathcal{P}_{0}\left(\one_{v}\right)\right)$.
The operators $T_{1}$ and $T_{2}$ are $q^{2}$- and $q$-regular,
respectively, and they are shown in \cite[Prop.\ 5.2]{Lubetzky2017RandomWalks}
to be collision-free. By Theorems \ref{thm:Col-Free-Ram} and \ref{thm:digraph-bound},
this implies 
\begin{align*}
\left\Vert \widetilde{T_{1}^{x}}|_{L_{0}^{2}\left(X_{1}^{1}\right)}\right\Vert _{2} & \leq\frac{1}{q^{2x}}{x+2 \choose 2}q^{4}\cdot q^{x-2}={x+2 \choose 2}q^{2-x},\\
\left\Vert \widetilde{T_{2}^{2y}}|_{L_{0}^{2}\left(X_{1}^{2}\right)}\right\Vert _{2} & \leq\frac{1}{q^{2y}}{2y+5 \choose 5}q^{5}\cdot\sqrt{q}^{2y-5}={2y+5 \choose 5}q^{5/2-y}.
\end{align*}
In addition, we have $\Vert\widetilde{T_{01}}|_{X}\Vert_{2}=1/\sqrt{q^{2}+q+1}$,
$\Vert\widetilde{T_{12}}|_{X}\Vert_{2}=1/\sqrt{q+1}$ and $\Vert\widetilde{T_{20}}|_{X}\Vert_{2}=\sqrt{(q^{2}+q+1)(q+1)}$
by degree considerations and evaluation on constant functions. Returning
to \eqref{eq:L2-decomposition}, we use $\left\Vert \cdot\right\Vert _{TV}\leq\frac{\sqrt{n}}{2}\left\Vert \cdot\right\Vert _{2}$
to conclude that
\begin{align*}
\big\Vert\mathcal{P}_{0}(\mu_{X}^{t,x,y})\big\Vert_{TV} & \leq\tfrac{\sqrt{n}}{2}\big\Vert\widetilde{T_{\left(x,y\right)}}\big|_{L_{0}^{2}\left(X^{0}\right)}\left(\mathcal{P}_{0}\left(\one_{v}\right)\right)\big\Vert_{2}\\
 & \leq\tfrac{\sqrt{n}}{2}\big\Vert\widetilde{T_{\left(x,y\right)}}\big|_{L_{0}^{2}\left(X^{0}\right)}\big\Vert_{2}\leq\frac{\sqrt{n}{\textstyle {x+2 \choose 2}}{\textstyle {2y+5 \choose 5}}q^{9/2}}{2q^{x+y}}.
\end{align*}
Taking now $r_{1}\leq x+y\leq t_{1}$, we assume $n$ is large enough
that $t_{1}\leq3r_{1}$, hence for $n$ large enough
\begin{align*}
\big\Vert\mathcal{P}_{0}(\mu_{X}^{t_{1},x,y})\big\Vert_{TV} & \leq\frac{\sqrt{n}{\textstyle {3r_{1}+2 \choose 2}}{\textstyle {6r_{1}+5 \choose 5}}q^{9/2}}{2q^{r_{1}}}\leq\frac{\sqrt{n}\left(7r_{1}\right)^{7}q^{9/2}}{2q^{r_{1}}}\\
 & =\frac{(7\log_{q^{2}}n+112\log_{q^{2}}\log_{q^{2}}n)^{7}q^{9/2}}{2(\log_{q^{2}}n)^{8}}\leq\varepsilon.\qedhere
\end{align*}
\end{proof}
We come to the proof of the main theorem of this section:
\begin{proof}[Proof of Theorem \ref{thm:cutoff-pgl3}]
From \eqref{eq:P0_Pcol}, Prop.\ \ref{prop:triangle-mix} (which
applies once $t_{1}\geq t_{\triangle}$), \eqref{eq:P0_bound}, and
Prop.\ \ref{prop:P0_xy_bound} we conclude that
\[
\left\Vert \mu_{X}^{t_{1}}-\pi_{X}\right\Vert _{TV}\leq3\varepsilon+\max_{r_{1}\leq x+y\leq t_{1}}\big\Vert\mathcal{P}_{0}(\mu_{X}^{t_{1},x,y})\big\Vert_{TV}\leq4\varepsilon,
\]
so that $t_{mix}(4\eps)\leq t_{1}$. Together with Prop.\ \ref{prop:mix-t0},
this implies the cutoff phenomenon at time $\frac{q^{2}+q+1}{q^{2}-1}\log_{q^{2}}n$,
with a window of size $O(\sqrt{\log_{q^{2}}n})$.
\end{proof}

\section{\label{sec:The-d-case}The case of $\textrm{PGL}{}_{d}$ for all
$d\protect\geq2$}

The main difference between $\mathrm{PGL}_{3}$ and the general case
is that $\mathrm{PGL}_{3}$ acts transitively on $\mathcal{B}^{j}$
for all $j$, but the same does not happen for general $d$. As a
result, the projected walk on the sector $\mathcal{S}=K\backslash\mathcal{B}$
is no longer isotropic - some directions are more likely to be chosen
than others. Our approach is to define a suitable metric on $\mathcal{S}$
and $\mathcal{B}$, which takes this asymmetry into account. Albeit,
$\mathrm{PGL}_{d}$ still acts transitively on $\mathcal{B}^{0}$,
so the 1-skeleton of $\mathcal{B}$ is a regular graph.

\subsection{The projected walk on $\mathcal{S}$}

As in Section \ref{sec:d=00003D3}, we consider a SRW $\left(\mathcal{X}_{t}\right)$
on $\mathcal{B}$ starting from $\xi$, which projects modulo $K$
to the weighted random walk on $\mathcal{S}$. Recalling the identification
$S\cong\mathbb{N}^{d-1}$, we define $x_{i}(t)$ to be the $i$-th
index of the projected walk $\Phi\left(\X_{t}\right)$, so that $\Phi\left(\mathcal{X}_{t}\right)=\vec{x}\left(t\right)=(x_{1}\left(t\right),\ldots,x_{d-1}\left(t\right)).$
We consider $\mathcal{S}$ as a weighted directed graph, with the
weight of an edge being the probability that the projected walk chooses
this edge. The weights are easier to describe outside the boundary:
it follows from the identification of the link of a vertex as the
flag complex of $\mathbb{F}_{q}^{d}$ that for every $\gamma\in\left\{ 0,1\right\} ^{d}\backslash\left\{ \zero,\one\right\} $
(see §\ref{subsec:Cartan-decomposition}) and $\varpi^{\alpha}\xi\in\mathcal{S\backslash\partial\mathcal{S}}$,
the probability of moving from $\varpi^{\alpha}\xi$ to $\varpi^{\alpha+\gamma}\xi$
is 
\[
\mathbb{P}\left[\varpi^{\alpha}\xi\rightarrow\varpi^{\alpha+\gamma}\xi\right]=\frac{q^{Z_{\gamma}}}{\deg\xi},\quad\text{where }Z_{\gamma}=\#\{(i,j)\ |\ i<j,\gamma_{i}=1,\gamma_{j}=0\}.
\]
At the boundary, the only difference is that $\gamma$ which leads
outside of $\mathcal{S}$ is folded back into it by the action of
the spherical Weyl group $S_{d}$.
\begin{claim}
\label{claim:qtimes}If $\Phi\left(\mathcal{X}_{t-1}\right)\notin\partial\mathbb{N}^{d-1}$,
then 
\[
\mathbb{P}\left[x_{i}(t)-x_{i}(t-1)=1\right]=q\cdot\mathbb{P}\left[x_{i}(t)-x_{i}(t-1)=-1\right]\qquad(1\leq i\leq d-1).
\]
\end{claim}

\begin{proof}
When moving along $\gamma$ (except at the boundary) the $i$-th index
$x_{i}$ of the projected walk changes by $\gamma_{i-1}-\gamma_{i}$.
The permutation $\tau$ on $\left\{ 0,1\right\} ^{d}$, which transposes
the $i$-th and $(i-1)$-th indices, induces an involution on $\left\{ 0,1\right\} ^{d}\backslash\left\{ \zero,\one\right\} $
that reverses the change in $x_{i}$. If $\gamma_{i-1}=1$ and $\gamma_{i}=0$
then $Z_{\gamma}=\frac{1}{q}Z_{\tau(\gamma)}$, so for every edge
that decreases $x_{i}$ there is an edge which increases it whose
weight is $q$ times larger.
\end{proof}
In what follows, for $\gamma\in\mathbb{Z}^{d}$ we denote by $\gamma'$
the difference vector
\[
\gamma'=(\gamma_{1}-\gamma_{2},\gamma_{2}-\gamma_{3},\dots,\gamma_{d-1}-\gamma_{d})\in\mathbb{Z}^{d-1}.
\]

\begin{prop}
\label{Prop:Transience2}The projected walk $\Phi\left(\mathcal{X}_{t}\right)$
visits $\partial\mathbb{N}^{d-1}$ only a finite number of times with
probability one.
\end{prop}

\begin{proof}
In essence this follows from the fact that the boundary is sink-less,
and on its complement the walk is a positively-drifted walk on $\mathbb{Z}^{d-1}$.
Namely, from any point in $\mathbb{N}^{d-1}$ the probability to enter
$D=\{\alpha\in\mathbb{N}^{d-1}\ |\ \forall i,\alpha_{i}\ge d-1\}$
in $d(d-1)$ steps is at least $\left(\deg\xi\right)^{-d(d-1)}$.
Let $P_{\vec{x}}$ be the probability that a walk which starts from
$\vec{x}\in\mathbb{N}^{d-1}$ ever touches the boundary. This is the
same as the probability of the walk on $\mathbb{Z}^{d-1}$ with transition
probability $\frac{q^{Z_{\gamma}}}{\deg\xi}$ of moving along $\gamma'$,
where $\gamma\in\left\{ 0,1\right\} ^{d}\backslash\left\{ \zero,\one\right\} $,
to ever reach from $\vec{x}$ to a point with a zero coordinate. This
is bounded by $\sum_{i=1}^{d-1}P_{\vec{x},i}$, where $P_{\vec{x},i}$
is the probability that the $i^{th}$ coordinate ever vanishes. But
on $\mathbb{Z}^{d-1}$ each coordinate is a drifted walk as in Claim
\ref{claim:qtimes}, hence it follows from standard arguments that
$P_{\vec{x},i}\leq\frac{1}{q^{x_{i}}}$. Thus, for $\vec{x}\in D$
we obtain $P_{\vec{x}}\le\frac{d}{q^{d}}<1$, and it follows that
the expected number of visits to the boundary is bounded by $\sum_{i=0}^{\infty}d^{i}\left(\deg\xi\right)^{d(d-1)}/q^{di}<\infty$.
\end{proof}

\subsection{Geometric operators on $\mathcal{B}(\mathrm{PGL}_{d})$}

For $1\leq j<d$, the geodesic $j$-flow $T_{j}$ defined in \cite{Lubetzky2017RandomWalks}
is a $q^{d-j}$-regular branching operator on $\mathcal{B}_{1}^{j}$
(the $j$-cells of type one), which takes the (ordered) cell $[v_{0},\ldots,v_{j}]$
to all cells $[v_{1},\ldots,v_{j},w]\in\mathcal{B}_{1}^{j}$ such
that $\left\{ v_{0},\ldots,v_{j},w\right\} \notin\mathcal{B}$. Defining
$\xi_{i}=\diag(\varpi^{\times(d-i)},1^{\times i})\xi$, and $\sigma^{j}=[\xi_{j},\xi_{j-1},\ldots,\xi_{0}]$
we have $\mathcal{B}_{1}^{j}=G\sigma^{j}$, and the operator $T_{j}$
corresponds to the double coset $P_{j}w_{j}P_{j}$, where
\[
P_{j}:=G_{\sigma^{j}}=\left\{ g\in K\,\middle|\,{g_{r,c}\in\varpi\mathcal{O}\text{ for}\atop c\leq\min\left(j,r-1\right)}\right\} ,\quad w_{j}=\left(\begin{array}{c|c|c}
 & ~\large{\mbox{\ensuremath{I_{j-1}}}}\vphantom{\Big)}~\\
\hline \varpi &  & 0\:\cdots\:0\\
\hline  &  & ~\large{\text{\ensuremath{I_{d-j}}}}\vphantom{\Big)}\ 
\end{array}\right)
\]
(note $\mathcal{B}_{1}^{0}=\mathcal{B}^{0}=G\sigma^{0}$ and $P_{0}=K$,
though there is no $0$-flow). Each double coset $P_{j}P_{j+1}$ ($0\leq j\leq d-2$)
gives via \eqref{eq:Hecke-op} an operator $T_{j,j+1}:\mathcal{B}_{1}^{j}\rightarrow\mathcal{B}_{1}^{j+1}$,
which takes $\sigma\in\mathcal{B}_{1}^{j}$ to all $v\sigma\in\mathcal{B}_{1}^{j+1}$
($v\in\mathcal{B}^{0}$). In addition, $P_{d-1}P_{0}$ yields $T_{d-1,0}:\mathcal{B}_{1}^{d-1}\rightarrow\mathcal{B}^{0}$
which returns the last vertex of a cell.
\begin{prop}
\label{prop:fiber-decomposition}For any $\vec{x}\in\mathbb{N}^{d-1}$,
the fiber $\Phi^{-1}\left(\vec{x}\right)$ equals $T_{\vec{x}}\left(\xi\right)$,
where
\[
T_{\vec{x}}:=T_{d-1,0}\prod\nolimits _{j=d-1}^{1}T_{j}^{jx_{j}}T_{j-1,j}:\mathcal{B}^{0}\rightarrow\mathcal{B}^{0}.
\]
\end{prop}

\begin{proof}
Denoting $g_{t}=\diag\left(\varpi^{x_{1}+\ldots+x_{t}},\varpi^{x_{2}+\ldots+x_{t}},\ldots,\varpi^{x_{t}},1,\ldots,1\right)$,
we claim that
\begin{equation}
T_{\vec{x}}\left(\xi\right)=Kg_{t-1}\left[\prod\nolimits _{j=t}^{d-1}P_{j}(w_{j}P_{j})^{jx_{j}}\right]\xi\qquad\text{for }1\leq t\leq d.\label{eq:Tx_t}
\end{equation}
For $t=1$, the definitions of $T_{\vec{x}}$ and the operators $T_{i},T_{i,j}$
indeed give
\[
T_{\vec{x}}\left(\xi\right)=\left[\prod\nolimits _{j=1}^{d-1}P_{j-1}P_{j}(P_{j}w_{j}P_{j})^{jx_{j}}\right]P_{d-1}P_{0}\xi=Kg_{0}\left[\prod\nolimits _{j=1}^{d-1}P_{j}(w_{j}P_{j})^{jx_{j}}\right]\xi.
\]
Assume that \eqref{eq:Tx_t} holds for some $1\leq t\leq d-1$. Explicit
computation as in \cite[§5.1]{Lubetzky2017RandomWalks} gives
\[
P_{t}\left(w_{t}P_{t}\right)^{tx_{t}}=\left(\begin{array}{c|c}
\varpi^{x_{t}}I_{t} & M_{t\times d-t}\left(\mathcal{O}\right)\\
\hline 0 & ~I_{d-t}
\end{array}\right)P_{t},
\]
and using $K$ to perform row elimination we obtain
\begin{align*}
T_{\vec{x}}\left(\xi\right) & =Kg_{t-1}\left(\begin{array}{c|c}
\varpi^{x_{t}}I_{t} & M_{t\times d-t}\left(\mathcal{O}\right)\\
\hline 0 & ~I_{d-t}
\end{array}\right)P_{t}\left[\prod\nolimits _{j=t+1}^{d-1}P_{j}(w_{j}P_{j})^{jx_{j}}\right]\xi\\
 & =Kg_{t}P_{t}P_{t+1}\left[\prod\nolimits _{j=t+1}^{d-1}P_{j}(w_{j}P_{j})^{jx_{j}}\right]\xi.
\end{align*}
Next, observe that $P_{t}P_{t+1}$ decomposes as $S_{t}P_{t+1}$ when
$S_{t}\subseteq K$ is any set which takes $\sigma^{t}$ to all $\left(t+1\right)$-cells
containing it. There are $(q^{d-t}-1)/(q-1)$ such cells, as in the
spherical building $\sigma^{t}$ corresponds to a $t$-dimensional
subspace of $\mathbb{F}_{q}^{d}$, and these cells to the minimal
subspaces containing it. This also shows how to compute such a transversal
$S_{t}$, and 
\[
S_{t}=\bigsqcup_{j=1}^{d-t}\diag\left(I_{t},Q_{j},I_{d-t-j}\right),\qquad\text{where }Q_{j}=\left(\begin{smallmatrix}\mathbb{F}_{q} & 1\\
\vdots &  & \ddots\\
\mathbb{F}_{q} &  &  & 1\\
1 & 0 & \cdots & 0
\end{smallmatrix}\right)\subseteq GL_{j}\left(\mathcal{O}\right),
\]
is one option. Since the matrices in $S_{t}$ above commute with $g_{t}$
(and lie in $K$), this shows that
\[
Kg_{t}P_{t}P_{t+1}\left[\prod\nolimits _{j=t+1}^{d-1}P_{j}(w_{j}P_{j})^{jx_{j}}\right]\xi=Kg_{t}\left[\prod\nolimits _{j=t+1}^{d-1}P_{j}(w_{j}P_{j})^{jx_{j}}\right]\xi,
\]
establishing \eqref{eq:Tx_t} for $t+1$. Taking $t=d$ in \eqref{eq:Tx_t}
we obtain 
\[
T_{\vec{x}}\left(\xi\right)=Kg_{d-1}\xi=\Phi^{-1}(x_{1},\ldots,x_{d-1}).\qedhere
\]
\end{proof}
The decomposition of $\Phi^{-1}\left(\vec{x}\right)$ suggests the
metric to impose on $\mathcal{S}$:
\begin{defn}
The $R$-norm on $\mathbb{N}^{d-1}\cong\mathcal{S}$ is
\[
R(x_{1},...,x_{d-1})=\sum\nolimits _{j=1}^{d-1}j(d-j)x_{i}.
\]
\end{defn}

In addition, we obtain a bound on the size of the fiber above $\vec{x}$:
\begin{cor}
\label{cor:FiberSizeGeneral_d}For $\vec{x}\in\mathbb{N}^{d-1}$,
the size of the fiber $\Phi^{-1}\left(\vec{x}\right)$ is bounded
by
\[
\left|\Phi^{-1}\left(\vec{x}\right)\right|\leq{\textstyle \prod\nolimits _{j=1}^{d-1}\frac{q^{j+1}-1}{q-1}\cdot q^{R\left(\vec{x}\right)}\leq d!\,q^{{d \choose 2}+R\left(\vec{x}\right)}}.
\]
\end{cor}

\begin{proof}
This follows from Prop.\ \ref{prop:fiber-decomposition}, as $T_{j}$
is $q^{d-j}$-regular, $T_{j,j+1}$ is $(q^{d-j}-1)/(q-1)$-regular
(see proof of Prop.\ \ref{prop:fiber-decomposition}), and $T_{d-1,0}$
is $1$-regular.
\end{proof}
Later, we will be interested in the long-term behavior of the $R$-distance
of the random walk on $\mathcal{S}$ from $\xi$. The change in the
$R$-distance distributes in the same manner whenever the walk is
not at the boundary $\partial\mathcal{S}$. We denote this distribution
by $\mathscr{D}$:
\[
\mathbb{P}\left[\mathscr{D}=j\right]=\sum\nolimits _{\gamma\in\left\{ 0,1\right\} ^{d}\backslash\left\{ \zero,\one\right\} :R\left(\gamma'\right)=j}\frac{q^{Z_{\gamma}}}{\deg\left(\xi\right)},
\]
and define $\mathscr{E}_{d}=\mathbb{E}[\mathscr{D}]$ and $\sigma_{d}=\sqrt{Var[\mathscr{D}]}$
(note that $\mathscr{E}$ from Section \ref{sec:d=00003D3} is $\mathscr{E}_{3}/2$).
The reciprocal of $\mathscr{E}_{d}$ is the constant $C_{d,q}$ which
appears in Theorem \ref{thm:main}: 
\begin{equation}
C_{d,q}:=\frac{1}{\mathscr{E}_{d}}=\left[\sum\nolimits _{\gamma\in\left\{ 0,1\right\} ^{d}\backslash\left\{ \zero,\one\right\} }\smash{\frac{R\left(\gamma'\right)q^{Z_{\gamma}}}{\deg\left(\xi\right)}}\right]^{-1}.\label{eq:Cdq}
\end{equation}

\begin{prop}
\label{prop:Ed-estimate}$\mathscr{E}_{d}=\left\lfloor \tfrac{d}{2}\right\rfloor \left\lceil \tfrac{d}{2}\right\rceil +O(\frac{1}{q})$.
\end{prop}

\begin{proof}
Recall that $\left[\begin{smallmatrix}d\\
j
\end{smallmatrix}\right]_{q}=\prod_{i=1}^{j}\frac{q^{d-i+1}-1}{q^{i}-1}$. Writing $f\approx g$ for $f\left(q\right)=g\left(q\right)\left(1+O(\nicefrac{1}{q})\right)$,
this implies $\left[\begin{smallmatrix}d\\
j
\end{smallmatrix}\right]_{q}\approx q^{j\left(d-j\right)}$, and thus
\[
\deg\left(\xi\right)\approx{\textstyle \sum\nolimits _{j=\left\lfloor d/2\right\rfloor }^{\left\lceil d/2\right\rceil }\left[\begin{smallmatrix}d\\
j
\end{smallmatrix}\right]_{q}\approx\frac{3-\left(-1\right)^{d}}{2}q^{\left\lfloor d/2\right\rfloor \left\lceil d/2\right\rceil }}.
\]
Similarly, $Z_{\gamma}$ is largest when $\gamma$ is a sequence of
$\left\lfloor d/2\right\rfloor $ or $\left\lceil d/2\right\rceil $
ones, followed by zeros. In this case we have $Z_{\gamma}=\left\lfloor d/2\right\rfloor \left\lceil d/2\right\rceil $,
and also $R\left(\gamma'\right)=\left\lfloor d/2\right\rfloor \left\lceil d/2\right\rceil $,
hence it follows from \eqref{eq:Cdq} that $\mathscr{E}_{d}\approx\left\lfloor d/2\right\rfloor \left\lceil d/2\right\rceil $.
\end{proof}
We demonstrate the first few cases of $\mathscr{E}_{d}$ and $\deg\left(\xi\right)$
in Table \ref{tab:deg_xi}.

\begin{table}
\begin{centering}
\begin{tabular}{|c|c|c|}
\hline 
$d$ & $\mathscr{E}_{d}\cdot\deg\left(\xi\right)$ & ${\textstyle \deg\left(\xi\right)=\sum\nolimits _{j=1}^{d-1}\left[\begin{smallmatrix}d\\
j
\end{smallmatrix}\right]_{q}}$\tabularnewline
\hline 
\hline 
2 & $q-1$ & $q+1$\tabularnewline
\hline 
3 & $4q^{2}-4$ & $2q^{2}+2q+2$\tabularnewline
\hline 
4 & $4q^{4}+8q^{3}+2q^{2}-4q-10$ & $q^{4}+3q^{3}+4q^{2}+3q+3$\tabularnewline
\hline 
5 & ${12q^{6}+8q^{5}+16q^{4}\atop +4q^{3}-8q^{2}-12q-20}$$\vphantom{\Big|}$ & ${2q^{6}+2q^{5}+6q^{4}\atop +6q^{3}+6q^{2}+4q+4}$\tabularnewline
\hline 
6 & ${9q^{9}+23q^{8}+22q^{7}+25q^{6}+21q^{5}\atop +3q^{4}-15q^{3}-28q^{2}-25q-35}$$\vphantom{\Big|}$ & ${q^{9}+3q^{8}+4q^{7}+7q^{6}+9q^{5}\atop +11q^{4}+9q^{3}+8q^{2}+5q+5}$\tabularnewline
\hline 
7 & ${24q^{12}+20q^{11}+52q^{10}+52q^{9}+56q^{8}+32q^{7}\atop +24q^{6}-8q^{5}-40q^{4}-52q^{3}-60q^{2}-44q-56}$$\vphantom{\Big|}$ & ${2q^{12}+2q^{11}+6q^{10}+8q^{9}+12q^{8}+12q^{7}\atop +18q^{6}+16q^{5}+16q^{4}+12q^{3}+10q^{2}+6q+6}$\tabularnewline
\hline 
\end{tabular}\\
\vspace*{0.5cm}
\par\end{centering}
\caption{\label{tab:deg_xi}The polynomials which arise in the computation
of $\deg\left(\xi\right),\mathscr{E}_{d},C_{d,q}$.}
\end{table}

\subsection{Cutoff on $\mathcal{B}(\mathrm{PGL}_{d})$}

Fix $\varepsilon>0$. For $r\ge0$ we define $B^{R}\left(\xi,r\right)$,
the $R$-normalized $r$-ball around $\xi$, to be the set of vertices
$\zeta\in\BT^{0}$ satisfying $R(\Phi(\zeta)))\leq r$. From Corollary
\ref{cor:FiberSizeGeneral_d} we obtain the bound
\begin{equation}
|B^{R}(\xi,r)|\leq\left|\left\{ \vec{x}\,|\,R\left(\vec{x}\right)\leq r\right\} \right|d!q^{\binom{d}{2}+r}\leq d!q^{\binom{d}{2}}\cdot r^{d-1}q^{r}.\label{eq:Ball_d_case}
\end{equation}
Defining
\[
r_{0}=\log_{q}n-d\log_{q}\log_{q}n,
\]
we obtain that for $n$ large enough
\begin{equation}
|B^{R}(\xi,r_{0})|\leq d!q^{\binom{d}{2}}\frac{\log_{q}^{d-1}(n)\cdot n}{\log_{q}^{d}(n)}<\varepsilon n.\label{eq:R-ball-bound}
\end{equation}

For a finite quotient $X$ of $\mathcal{B}$ and $v\in X^{0}$, we
choose a covering map $\varphi:\mathcal{B}\rightarrow X$ with $\varphi\left(\xi\right)=v$
as before, and define $B^{R}\left(v,r\right)=\varphi\left(B^{R}\left(\xi,r\right)\right)$
(this is independent of the choice of $\varphi$ as $R$ is $K$-invariant).
As in Section \ref{sec:d=00003D3}, $\left(X_{t}\right)=\varphi\left(\mathcal{X}_{t}\right)$
is a SRW on $X$ starting from $v$. We define 
\begin{align*}
\rho(t) & =R(\Phi(\mathcal{X}_{t}))=\sum_{j=1}^{d-1}j\left(d-j\right)x_{j}\left(t\right),\\
\end{align*}
and recall that $\rho\left(t\right)-\rho\left(t-1\right)\sim\mathscr{D}$
when $\Phi\left(\mathcal{X}_{t-1}\right)\notin\partial\mathbb{N}^{d-1}$.
By the same arguments as in $\mathrm{PGL}_{3}$ (with Prop.\ \ref{Prop:Transience2}
replacing Prop.\ \ref{prop:PGL3_transient}), we see that $\left(\rho(t)-\mathscr{E}_{d}t\right)/(\sigma_{d}\sqrt{t})\Rightarrow\mathcal{N}(0,1)$.
\begin{prop}
\label{prop:mix-t0-1}There exists $s=s\left(q,\varepsilon\right)$
such that $t_{mix}(1-3\eps)>t_{0}$ for large enough $X$, where
\[
t_{0}=\frac{1}{\mathscr{E}_{d}}\log_{q}n-(s+1)\sqrt{\log_{q}n}.
\]
\end{prop}

\begin{proof}
Using a similar computation to the one in Prop.\ \ref{prop:t0-to-r0}
we obtain $\mathbb{P}\left[\rho(t_{0})>r_{0}\right]<\mathbb{P}[Z>cs]+\varepsilon$
for $Z\sim\mathcal{N}\left(0,1\right)$ and $c=c(q,d)=\mathscr{E}_{d}^{3/2}/\sigma_{d}$.
Combining this with \eqref{eq:R-ball-bound}, the proof continues
as that of Prop.\ \ref{prop:mix-t0}, with $B^{R}(v,r_{0})$ replacing
$B(v,r_{0})$. 
\end{proof}
We turn to the upper bound, starting again with the trivial spectrum.
For $X=\Gamma\backslash\mathcal{B}$, we have $\left\{ \mathrm{ord}_{\varpi}\det\gamma\,\middle|\,\gamma\in\Gamma\right\} =m\mathbb{Z}$
for a unique $m\mid d$, and we say that $X$ is $m$-partite. We
obtain a map $\col:X^{0}\rightarrow\mathbb{Z}/m\mathbb{Z}$, which
we again consider as a simplicial map from $X$ to $\triangle_{m-1}$,
the $\left(m-1\right)$-dimensional simplex. We have $L_{\col}^{2}\left(X^{0}\right)=\col^{-1}\left(L^{2}\left(\triangle_{m-1}^{0}\right)\right)$,
and $L_{0}^{2}\left(X^{0}\right)$,$\mathcal{P}_{\col}$, $\mathcal{P}_{0}$
are defined as before. The walk induced from $X$ on $\triangle_{m-1}$
is not simple, but every edge is taken with positive probability.
Furthermore, unless $d=m=2$, the walk is aperiodic, since even if
$m=2$ there are loops at the vertices of $\triangle_{1}$ when $d\geq3$.
The case $d=m=2$ is that of bipartite Ramanujan graphs, on which
SRW does not mix, and for the rest of the paper we exclude this case.
We conclude as before that there exists $t_{\triangle}=t_{\triangle}\left(\varepsilon\right)$
with $\left\Vert \mathcal{P}_{\!\col}(\mu_{X}^{t})-\pi_{X}\right\Vert _{TV}\leq\varepsilon$
for any $t\geq t_{\triangle}$, hence
\begin{equation}
\begin{alignedat}{1}\left\Vert \mu_{X}^{t}-\pi_{X}\right\Vert _{TV} & \leq\left\Vert \mathcal{P}_{0}(\mu_{X}^{t})\right\Vert _{TV}+\left\Vert \mathcal{P}_{\!\col}(\mu_{X}^{t})-\pi_{X}\right\Vert _{TV}\\
 & \leq\left\Vert \mathcal{P}_{0}(\mu_{X}^{t})\right\Vert _{TV}+\varepsilon.
\end{alignedat}
\label{eq:P_0+epsilon}
\end{equation}
We now choose 
\begin{align*}
r_{1} & =\log_{q}n+4\,d!\log_{q}\log_{q}n,\\
t_{1} & =\frac{1}{\mathscr{E}_{d}}\log_{q}n+(s+1)\sqrt{\log_{q}n},
\end{align*}
and the same $c$ and $s$ as in Prop.\ \ref{prop:mix-t0-1} give
for $n$ large enough
\[
\mathbb{P}\left[\rho(t_{1})<r_{1}\right]\leq\mathbb{P}[Z>cs]+\varepsilon<2\varepsilon.
\]
Denoting $\mu_{X}^{t,\vec{x}}=\mathbb{P}[X_{t}\!=\cdot\,\big|\,\vec{x}(t)\!=\!\vec{x}]$
and 
\[
S=\left\{ \vec{x}\in\mathbb{N}^{d-1}\,\middle|\,{\textstyle \sum_{i=1}^{d-1}x_{i}\leq t_{1}}\text{ and }R(\vec{x})\geq r_{1}\right\} ,
\]
we obtain
\begin{align}
\left\Vert \mathcal{P}_{0}(\mu_{X}^{t_{1}})\right\Vert _{TV} & =\Big\Vert\mathbb{P}\left[\rho(t_{1})<r_{1}\right]\mathcal{P}_{0}\left(\mathbb{P}[X_{t}=\cdot\,\middle|\,\rho(t_{1})<r_{1}]\right)\nonumber \\
 & \qquad\qquad+\sum\nolimits _{\vec{x}\,:\,r_{1}\leq R(\vec{x})}\mathbb{P}\big[\vec{x}(t_{1})=\vec{x}\big]\mathcal{P}_{0}(\mu_{X}^{t_{1},\vec{x}})\Big\Vert_{TV}\label{eq: Total_Prob_d_case}\\
 & \leq2\varepsilon+\max_{\vec{x}\in S}\big\Vert\mathcal{P}_{0}(\mu_{X}^{t_{1},\vec{x}})\big\Vert_{TV}.\nonumber 
\end{align}

\begin{prop}
\label{prop:P0_xy_bound-1}If $\sum_{i=1}^{d-1}x_{i}\leq t_{1}$ and
$R(\vec{x})\geq r_{1}$, then for $n$ large enough $\big\Vert\mathcal{P}_{0}(\mu_{X}^{t_{1},\vec{x}})\big\Vert_{TV}\leq\varepsilon$.
\end{prop}

\begin{proof}
Denote by $\pi_{\vec{x}}$ the uniform distribution on $\Phi^{-1}\left(\vec{x}\right)$.
Using the same argument as in Prop.\ \ref{prop:P0_xy_bound}, with
Prop.\ \ref{prop:fiber-decomposition} replacing Prop.\ \ref{prop:fiber-by-walks},
we obtain (for any $t$)
\[
\mu_{X}^{t,\vec{x}}=\pi_{\vec{x}}\circ\varphi^{-1}=\widetilde{T_{\vec{x}}}\left(\one_{\xi}\right)\circ\varphi^{-1}=\widetilde{T_{\vec{x}}}\big|_{X}\left(\one_{v}\right)=\widetilde{T_{d-1,0}}\prod\nolimits _{j=d-1}^{1}\widetilde{T_{j}}^{jx_{j}}\widetilde{T_{j-1,j}}\big|_{X}\left(\one_{v}\right).
\]
Again the operators $\widetilde{T_{i}}|_{X}$ and $\widetilde{T_{ij}}|_{X}$
decompose with respect to $L^{2}=L_{\col}^{2}\oplus L_{0}^{2}$, so
that $\mathcal{P}_{0}(\mu_{X}^{t,\vec{x}})=\widetilde{T_{\vec{x}}}\big|_{X}\left(\mathcal{P}_{0}\left(\one_{v}\right)\right)$.
By \cite[§5.1]{Lubetzky2017RandomWalks}, the $j$-flow operator $T_{j}$
is $q^{d-j}$-regular and collision-free, and using Theorem \ref{thm:Col-Free-Ram}
and \eqref{eq:Ram-dig-power} we obtain
\begin{align*}
\left\Vert \widetilde{T_{j}^{jx_{j}}}|_{L_{0}^{2}\left(X_{1}^{j}\right)}\right\Vert _{2} & \leq\frac{1}{q^{\left(d-j\right)jx_{j}}}\cdot\left(jx_{j}+\left(d\right)_{j}\right)^{\left(d\right)_{j}}\left(q^{d-j}\right)^{\frac{jx_{j}+\left(d\right)_{j}}{2}}\\
 & =\frac{\left(q^{(d-j)/2}\left(jx_{j}+(d)_{j}\right)\right)^{\left(d\right)_{j}}}{q^{\left(d-j\right)jx_{j}/2}}\overset{{\scriptscriptstyle (*)}}{\leq}\frac{\left(2q^{d}dt_{1}\right)^{\left(d\right)_{j}}}{q^{\left(d-j\right)jx_{j}/2}},
\end{align*}
where $\left(*\right)$ assumes $n$ is large enough. If $T$ is a
branching operator of out-degree $d_{o}$ and in-degree $d_{i}$,
then $\Vert\widetilde{T}\Vert_{2}=\sqrt{d_{i}/d_{o}}$, so that $\left\Vert \smash{\widetilde{T_{{\scriptscriptstyle d-1,0}}}}\right\Vert _{2}\prod\nolimits _{j=d-1}^{1}\left\Vert \smash{\widetilde{T_{{\scriptscriptstyle j-1,j}}}}\right\Vert _{2}=1$.
Using $\sum_{j=1}^{d-1}\left(d\right)_{j}<2\cdot d!$, we obtain that
for $n$ large enough 
\[
\left\Vert \widetilde{T_{\vec{x}}}|_{L_{0}^{2}\left(X^{0}\right)}\right\Vert _{2}\leq\prod_{j=1}^{d-1}\left\Vert \widetilde{T_{j}^{jx_{j}}}|_{L_{0}^{2}\left(X_{1}^{j}\right)}\right\Vert _{2}\leq{\textstyle \frac{1}{q^{R\left(\vec{x}\right)/2}}}\prod_{j=1}^{d-1}(2q^{d}dt_{1})^{\left(d\right)_{j}}\leq{\textstyle \frac{\left(2q^{d}dt_{1}\right)^{2d!-1}}{q^{R\left(\vec{x}\right)/2}}}.
\]
Taking $n$ large enough that $t_{1}\leq\frac{2}{\mathscr{E}_{d}}\log_{q}n$,
we obtain from $R\left(\vec{x}\right)\geq r_{1}$ that for $n$ large
enough 
\[
\big\Vert\mathcal{P}_{0}(\mu_{X}^{t,\vec{x}})\big\Vert_{TV}=\big\Vert\widetilde{T_{\vec{x}}}\big|_{X}\left(\mathcal{P}_{0}\left(\one_{v}\right)\right)\big\Vert_{TV}\leq{\textstyle \frac{\sqrt{n}}{2}}\left\Vert \widetilde{T_{\vec{x}}}\big|_{L_{0}^{2}\left(X^{0}\right)}\right\Vert _{2}\leq{\textstyle \frac{\left(4q^{d}d\mathscr{E}_{d}^{-1}\log_{q}n\right)^{2d!-1}}{2(\log_{q}n)^{2d!}}<\varepsilon.}\qedhere
\]
\end{proof}
We conclude with the proof of the main theorem:
\begin{proof}[Proof of Theorem \ref{thm:main}]
From (\ref{eq:P_0+epsilon}), \eqref{eq: Total_Prob_d_case}, and
Prop.\ \ref{prop:P0_xy_bound-1} we conclude that $t_{mix}(4\eps)\leq t_{1}$
for $n=|X^{0}|$ large enough. Together with Prop.\ \ref{prop:mix-t0-1},
this implies the cutoff phenomenon at time $\frac{1}{\mathscr{E}_{d}}\log_{q}n$,
with a window of size $O(\sqrt{\log n})$, and $C_{d,q}=\frac{1}{\mathscr{E}_{d}}=\tfrac{1}{\left\lfloor d/2\right\rfloor \left\lceil d/2\right\rceil }+O(\tfrac{1}{q})$
by Prop.\ \ref{prop:Ed-estimate}.
\end{proof}
\bibliographystyle{amsalpha}
\bibliography{mybib}

\noindent \begin{flushleft}
\noun{\small{}Einstein Institute of Mathematics, Hebrew University,
Israel.}\texttt{}~\\
\texttt{\small{}michael.chapman@mail.huji.ac.il, parzan@math.huji.ac.il.}{\small\par}
\par\end{flushleft}
\end{document}